\documentclass[12pt]{amsart}

\usepackage{amstext}
\usepackage{amsfonts}
\usepackage{amsthm}
\usepackage{amsmath}
\usepackage{amssymb}
\usepackage{latexsym}
\usepackage{amsfonts}
\usepackage{graphicx}
\usepackage{texdraw}
\usepackage{graphpap}
\usepackage{subfigure}
\usepackage{enumerate}

\usepackage[pagebackref,hypertexnames=false, colorlinks, citecolor=black, linkcolor=black]{hyperref} 
\usepackage[backrefs]{amsrefs}

       \newcommand{\fr}{\frac}

       \newcommand{\beqn}{\begin{eqnarray}}
       \newcommand{\eqn}{\end{eqnarray}}
       \newcommand{\beqm}{\begin{eqnarray*}}
       \newcommand{\eqm}{\end{eqnarray*}}
       \newcommand{\bequ}{\begin{equation}}
       \newcommand{\equ}{\end{equation}}

\newcounter{totenum}
\let\Oldenumerate\enumerate
\def\enumerate{\refstepcounter{totenum}\Oldenumerate}
\newcommand{\myitem}[1]{\refstepcounter{enumi}\item[\hypertarget{#1}{\normalfont{(\arabic{enumi})}}]}

\numberwithin{equation}{section}

\newtheorem{thm}{Theorem}[section]
\newtheorem{lm}[thm]{Lemma}
\newtheorem{cor}[thm]{Corollary}
\newtheorem{defin}[thm]{Definition}

\newtheorem{prob}[thm]{Problem}
\newtheorem{prop}[thm]{Proposition}
\newtheorem*{prop*}{Proposition}

\begin{document}

\title[Localization  and compactness on  Fock Spaces]{Localization  and compactness of  Operators on  Fock Spaces}

\author[Z. Hu]{Zhangjian Hu$^1$}
\address{Z. Hu, Department of Mathematics, Huzhou University, Huzhou, Zhejiang,    313000, China}
\email{huzj@zjhu.edu.cn}

\author[X. Lv]{Xiaofen Lv$^1$}
\address{X. Lv, Department of Mathematics, Huzhou University, Huzhou, Zhejiang,    313000, China}
\email{lvxf@zjhu.edu.cn}

\thanks{1. This research is partially supported by the National Natural Science Foundation of China (11601149, 11771139, 11571105), Natural Science Foundation of Zhejiang province (LY15A010014).}

\author[B. D. Wick]{Brett D. Wick$^2$}
\address{B. D. Wick, Department of Mathematics, Washington University - St. Louis, One Brookings Drive, St. Louis, MO USA 63110.}
\email{wick@math.wustl.edu}

\thanks{2. Research supported in part by a National Science Foundation DMS grant \#0955432.}

 \date{\today}


\begin{abstract}
 For $0<p\leq\infty$, let $F^{p}_\varphi$ be the   Fock space induced by a weight function $\varphi$ satisfying $ dd^c \varphi \simeq \omega_0$.  In this paper, given $p\in (0, 1]$ we introduce the concept of  weakly localized operators on $ F^{p}_\varphi$, we  characterize the compact operators in the algebra generated by  weakly localized operators. As an application, for $0<p<\infty$  we prove that an operator $T$ in the algebra generated by bounded Toeplitz operators with $\textrm{BMO}$ symbols is compact on $F^p_\varphi$ if and only if its Berezin transform satisfies certain vanishing property at $\infty$.  In the classical Fock space, we extend the Axler-Zheng condition on linear operators $T$, which ensures $T$ is compact on $F^p_{\alpha}$ for all possible $0<p<\infty$.
\end{abstract}

\maketitle

\section{Introduction}

 Let $H(\mathbb{C}^n)$ be the collection of all entire functions on $\mathbb{C}^n$, and let  $\omega_0= dd^c |z|^2$  be  the Euclidean
 K\"{a}hler form on $\mathbb{C}^n$, where $d^c=\fr {\sqrt{-1}}{4} (\overline{\partial} -\partial)$. Set $B(z,r)$ to be  the Euclidean ball in $\mathbb{C}^n$ with center $z$ and
 radius $r$, and $B(z,r)^c=\mathbb{C}^n\backslash B(z,r)$.  Throughout the paper, we assume that  $\varphi\in C^2(\mathbb{C}^n)$ is real-valued and  there are two positive numbers $M_1, M_2$ such that
\begin{equation}
\label{e:1.1}
         M_1 \omega_0 \le dd^c \varphi \le M_2 \omega_0
\end{equation}
 in the sense of currents. The expression \eqref{e:1.1} will be denoted as $ dd^c \varphi \simeq \omega_0$.
Given $0<p<\infty$ and a positive Borel measure $\mu$ on $\mathbb{C}^n$, let
  $L^p_\varphi(\mu)$ be the space defined by
$$
          L^p_\varphi(\mu)= \left\{ f \textrm{ is } \mu\textrm{-measurable on } \mathbb{C}^n:  f(\cdot)e^{-\varphi(\cdot) } \in L^p(\mathbb{C}^n,  d\mu)\right\}.
$$
When $d\mu =dV$, the Lebesgue measure on $\mathbb{C}^n$, we write   $L^p_\varphi$ for $ L^p_\varphi(\mu)$ and  set
$$
        \|f\|_{p, \varphi}= \left (\int_{\mathbb{C}^n} \left| f(z) e^{-\varphi(z)}\right|^p dV(z)\right)^{\fr 1p}.
$$
For $0<p<\infty$  the Fock space $F^p_\varphi$ is defined as $F^p_\varphi = L^p_\varphi \cap  H( \mathbb{C}^n)$,
and
 $$
 F^\infty_\varphi=\left\{f\in H(\mathbb{C}^n): \|f\|_{\infty, \varphi}= \sup\limits_{z\in\mathbb{C}^n}\left| f(z)\right| e^{-\varphi(z)}<\infty\right\}.
 $$
 $F^p_\varphi$ is a Banach space with norm $\|\cdot\|_{p, \varphi}$ when $1\le p\leq\infty$ and   $F^p_\varphi$ is
 a  Fr\'{e}chet space with  distance $\rho (f, g)=\|f-g\|_{p, \varphi}^p$ if $0<p<1$. The typical model of  $\varphi$ is $\varphi(z)= \fr \alpha 2 |z|^2$ with $\alpha>0$, which induces the classical Fock space. For this particular special weight $\varphi$,  $F^p_\varphi$ and $\|\cdot\|_{p,\varphi}$  will be
  written  as $F^p_\alpha$ and $\|\cdot\|_{p,\alpha}$, respectively.
 The  space $F^p_\alpha$ has been studied by many authors, see
  \cites{BI12, CIL11, HL11,WX16, WCZ12, XZ13, Zh12} and the references therein. Another special case is with $\varphi(z)= \fr \alpha 2 |z|^2- \fr m 2 \ln (A+|z|^2)$ with suitable $A>0$, and then $F^p_\varphi$ is the Fock-Sobolev space $F^{p, m}_\alpha$ studied in \cites{CIJ14, CZ12}.

It is  well-known that $F^2_\varphi$ is a Hilbert space with inner product
$$
\langle f, g\rangle_{F^2_\varphi}=\int_{\mathbb{C}^n}   f(z) \overline{g(z)}  e^{-2\varphi(z)}   dV(z).
$$
Given $z,w\in \mathbb{C}^n$, the reproducing kernel of  $F^2_\varphi$
 will be denoted by
 $K_z(w)=K(w, z)$. We write $k_z= \frac{K_z}{\|K_z\|_{2, \varphi}}$ to denote the normalized reproducing  kernel.  Given some bounded linear operator $T$ on $F^p_\varphi$,   the Berezin transform of $T$ is well defined as
$$
            \widetilde{T}(z)=  \langle T  k_z, k_z \rangle_{F^2_\varphi},
$$
since $T  k_z\in F^{p}_\varphi\subset F^{\infty}_\varphi$ and $k_z\in F^{1}_\varphi$. Set $P$ to be  the projection from $L^2_\varphi $ to $F^2_\varphi$, that is
 $$
 Pf(z)= \int _{\mathbb{C}^n} f(w) K(z, w) e^{-2\varphi(w)} dV(w)  \verb#  # \textrm{ for } f\in L^2_\varphi.
 $$
   For a complex Borel measure $\mu$ on $\mathbb{C}^n$ and $f\in F^p_\varphi$, we define the Toeplitz operator $T_\mu$ to be
  $$
             T_\mu f(z) = \int_{\mathbb{C}^n} f(w) K(z, w) e^{-2\varphi(w)} d\mu(w).
   $$
If $d\mu= g dV$, for short, we will use  $T_g$ to stand for the induced Toeplitz operator and will use that  $\widetilde{g}=\widetilde{T_g}$.
\newline


In the case of Fock spaces $F^2_\alpha$,  fixed $g$ bounded on $\mathbb{C}^n$, $|\langle T_g k_z, k_w \rangle|$  as a function of $(z, w)$ decays very fast off the diagonal of $\mathbb{C}^n\times \mathbb{C}^n$, see \cite{XZ13}*{Proposition 4.1}. From this point of view,
 Xia and Zheng in \cite{XZ13}  introduced the notion  of ``sufficiently localized'' operators on $F^2_\alpha$ which include  the  algebra  generated by   Toeplitz operators with bounded symbols, and they proved that, if $T$ is in the $C^*$-algebra generated by the class of sufficiently  localized  operators, $T$ is compact on $F^2_\alpha$ if and only if its Berezin transform tends to zero when $z$ goes to infinity. In \cite{Is13}, Isralowitz extended \cite{XZ13} to the generalized Fock space $F^2_\varphi$  with $d d^c \varphi \simeq \omega_0$. Isralowitz, Mitkovski and
the third author extended Xia and Zheng's idea further in \cite{IMW13}  to what they called ``weakly localized'' operators on $F^p_\varphi$ with $1<p<\infty$. They showed that,  if $T$ is in the $C^*$-algebra generated by the class of  weakly  localized  operators,   $T$ is compact on $F^p_\varphi$ if and only if its Berezin transform shares certain  vanishing property  near infinity.
We would like to emphasize that the prior results in the area, for example \cites{AZ98, BI12, CIL11, HL14, Is11, Is13, MSW13, MW12, Su07, WCZ12, XZ13, Zo03, Zh12}, depend strongly on two points. The  first is the use of Weyl unitary operators induced by holomorphic self mappings of the domain; and the second is the restriction on the  range of the exponent $p$, for example $p=2$ or $1<p<\infty$, so that Banach space techniques are applicable. But on $F^p_\varphi$ with $0<p<1$ and $ dd^c \varphi \simeq \omega_0$  these two points are not available.
\newline

The main purpose of this work is, on $F^p_\varphi$ with $0<p<1$ and $ dd^c \varphi \simeq \omega_0$, to study the so called ``weakly localized" operators $\textrm{WL}^\varphi_p$  and to characterize those compact operators $T\in \textrm{WL}^\varphi_p$.
The paper is divided into four sections. In Section \ref{s:two},   we introduce the concept of weakly localized operators $\textrm{WL}^\varphi_p$ for  $0<p\leq 1$,  we will characterize the compact operators in ${\textrm{WL}}_p^\varphi$,  and furthermore give a quantity equivalent to   the essential norm of an operator in ${\textrm {WL}}_p^\varphi$.
         Section \ref{s:three} is devoted to the compactness of Toeplitz operators induced by   $\textrm{BMO}$
symbols acting on $F^p_{\varphi}$ for all $0<p<\infty$, our theorem shows an operator $T$ in the  algebra generated by bounded Toeplitz operators with $\textrm{BMO}$ symbols is compact on $F^p_\varphi$ if and only if its Berezin transform satisfies a certain vanishing property at $\infty$ (more precisely, $\lim\limits_{z\to \infty} \widetilde{T}(z)=0$ when $\varphi(z)= \fr \alpha 2 |z|^2$). In Section  \ref{s:four}, we  extend Axler-Zheng's condition on linear operators $T$, which insures $T$ are bounded (or compact) on  $F^p_\alpha$ for all possible $0<p<\infty$. In the final
section, we provide some remarks and point to some open problems.
\newline

In what follows, $C$ will denote a positive constant whose value may change from one occasion to another but does not depend on the functions or operators in consideration.  For two positive quantities $A$ and $B$, the expression $A\simeq B$ means there is some $C>0$ such that $\fr 1C B\le A\le C B$.

\section{The operator class $\textrm{WL}^\varphi_p$ with $0<p\le 1$}
\label{s:two}

As a generalization of the ``strongly localized'' operators of Xia and Zheng in \cite{XZ13}, Isralowitz, Mitkovski and the third author introduced  ``weakly localized'' operators on $F^p_\varphi$ with $1<p<\infty$, see \cite{IMW13}. In this section, we first  give the definition of weakly localized operators  on  $F^p_\varphi$ when $0<p\le 1$.
 We use $\mathcal D$ to stand for the  linear span of all normalized  reproducing kernel functions $k_z(\cdot)$. It is obvious that $\mathcal D$ is dense in $F^p_\varphi$.   As in \cite{IMW13}, we will assume that the domain of every linear operator  $T$ appearing in this paper contains $ \mathcal D $, and that the function $z\mapsto T K_z$ is conjugate holomorphic.
We also assume  the range of   $T$ is in  $F^\infty_{\varphi}$. Then $\langle Tk_z,
  k_w\rangle_{F^2_\varphi}$ can make sense.

\begin{defin}
\label{d:2.1}
 Let $0<p<\infty$, set $s=\min\{1, p\}$. A linear operator $T$ from
 $\mathcal D$ to $F^\infty_\varphi$ is called
 weakly localized for $F_\varphi^p$ if
 \begin{equation}
 \label{e:2.1}
  \sup_{z\in \mathbb{C}^n}  \int_{\mathbb{C}^n} \left |\langle Tk_z,
  k_w\rangle_{F^2_\varphi}\right|^s
  dV(w)<\infty, \verb#  #  \sup_{z\in \mathbb{C}^n}  \int_{\mathbb{C}^n} \left |\langle k_z,
  Tk_w\rangle_{F^2_\varphi}\right|^s
  dV(w)<\infty;
\end{equation}
and
\begin{equation}
\label{e:2.2}
   \lim_ {r\to \infty}   \sup_{z\in \mathbb{C}^n}  \int_{B(z,r)^c} \left |\langle Tk_z,
   k_w\rangle_{F^2_\varphi}\right|^s
  dV(w)=0,
\end{equation}
\begin{equation}
\label{e:2.3}
\lim_ {r\to \infty}  \sup_{z\in \mathbb{C}^n}  \int_{B(z,r)^c}  \left |\langle k_z,
  Tk_w\rangle_{F^2_\varphi}\right|^s
  dV(w)=0.
\end{equation}
The  algebra generated by   weakly localized operators for $F^p_\varphi$ will be denoted by $\textrm{WL}_p^\varphi$.
For  $\varphi(z)= \fr \alpha 2 |z|^2$, we write $\textrm{WL}_p^\varphi=\textrm{WL}_p^\alpha$ for convenience.
\end{defin}

 When $1\leq p<\infty$   $\textrm{WL}_p^\varphi=\textrm{WL}_1^\varphi$ by  definition, and then  Definition \ref{d:2.1} was first introduced in \cite{IMW13}.   Let ${\mathcal T}^\varphi_p$ denote  the  Toeplitz algebra on $F_p^\varphi$ generated by $L^\infty$  symbols,
and let $ {\mathcal K}(F^p_\varphi)$ be the set of all compact operators on $F_p^\varphi$. We use $\|T\|_{e, F^p_\varphi} $ to stand for the essential norm of a given operator $T$ on $F^p_\varphi$
 $$
 \|T\|_{e, F^p_\varphi} = \inf\left\{ \|T-A\|_{F^p_\varphi\to F^p_\varphi} : A \in {\mathcal K}(F^p_\varphi)\right\}.
 $$
The  purpose  of  this section is to  characterize compact   operators  in $\textrm{WL}^\varphi_p$, $0<p\le 1$.
To carry out our analysis, we need some preliminary facts.
\newline

\begin{lm}[\cite{SV12}]
\label{l:2.2}
Given $\varphi$ as in the introduction, the Bergman kernel $K(\cdot, \cdot)$ for  $F^2_\varphi$ satisfies the following estimates:
\begin{enumerate}
\myitem{l:2.2e1} There exists $C$ and $ \theta>0$ such that
 $$
 |K(z,w)|e^{-\varphi(z)}e^{-\varphi(w)} \leq Ce^{-\theta|z-w|}  \   \textrm{  for } z,w\in \mathbb{C}^n.
 $$

\myitem{l:2.2e2} There exists some  $r>0$ such that
$$
\left|K(z,w)\right|e^{-\varphi(z)}e^{-\varphi(w)}\simeq 1  \  \textrm{ whenever  }  w\in B\left(z, r\right)   \textrm{ and  }  z\in\mathbb{C}^n.
 $$

\myitem{l:2.2e3} For  $0<p\leq\infty$ fixed,
 $$\|K(\cdot, z)\|_{p, \varphi}\simeq e^{\varphi(z)}\simeq \sqrt{K(z,z)},\verb#  # z\in \mathbb{C}^n. $$
\end{enumerate}
\end{lm}

\begin{lm}[\cite{HL14}]
\label{l:2.3}
Suppose $0<p<\infty$ and $r>0$. Then there exists $C$ such that for  $f\in
H(\mathbb{C}^n)$ and $z\in \mathbb{C}^n$, we have
$$
\left|f(z)e^{-\varphi(z)}\right|^p\leq
C\int_{B(z,r)}\left|f(w)e^{-\varphi(w)}\right|^pdV(w)
$$
and
$$\int_{\mathbb{C}^n}\left|f(z)e^{-\varphi(z)}\right|^pd\mu(z)\leq C\int_{\mathbb{C}^n}\left|f(z)e^{-\varphi(z)}\right|^p\widehat{\mu}_{r}(z)dV(z)$$
where $\mu$ is some given positive Borel measure and  $\widehat{\mu}_r(\cdot)=\frac{\mu(B(\cdot,r))}{{V\left(B(\cdot,r)\right)}}$.
\end{lm}

Let $d(\cdot, \cdot)$   be the Euclidean distance on $\mathbb{C}^n$.   Given some domain $\Omega \subseteq \mathbb{C}^n$,
write $\Omega^+ =\{z\in \mathbb{C}^n: d(z, \Omega)<1\}$, and $\Omega^+$  is again a domain.  Set   $ {\mathcal L}=\left\{a+bi: a, b\in \fr 14 {\mathbb{Z}}^n\right\}$,  ${\mathcal L}$ is countable so that
 we may write ${\mathcal L}=\{z_1, z_2, \cdots, z_j, \cdots\}$. It is obvious that ${\mathcal L}$ forms a $1/4$-lattice in $\mathbb{C}^n$ (see \cite{Zh12} for the definition). For $E\subset \mathbb{C}^n$,  let $\chi_E$ be the characteristic function of $E$.
  We have some absolute constant $N>0$ such that
\begin{equation}
\label{e:2.4}
      \sum_ {z_j\in {\mathcal L}} \chi_ {B(z_j,  \frac{1}{2})   }(w)  \le N \verb#   # \textrm{ for } w\in \mathbb{C}^n.
\end{equation}

\begin{lm}
\label{l:2.4}
 For $0<p\le 1$  there is some constant $C$ (depending only on $p$ and $n$) such that for any  domain $\Omega\subset \mathbb{C}^n$ and    $f\in H(\mathbb{C}^n)$,
$$
\left(  \int_\Omega \left|f(w)e^{-\varphi(w)}\right| dV(w) \right)^p \le  C  \int_{\Omega^+} \left|f(w)e^{-\varphi(w)}\right|^p  dV(w).
$$
\end{lm}

\begin{proof}
It is trivial to see that  $(u+v)^p\le u^p+v^p$ for   positive $u, v$ and $0<p\le 1$. Applying Lemma \ref{l:2.3} and \eqref{e:2.4},  for $f\in H(\mathbb{C}^n)$ we have
\beqm
\left(  \int_\Omega \left|f(w)e^{-\varphi(w)}\right| dV(w) \right)^p & \le &   \left( \sum_ {z_j\in {\mathcal L}} \int_{\Omega \cap B(z_j,   1/4 )} \left|f(w)e^{-\varphi(w)}\right| dV(w) \right)^p\\
    & \le &  C\sum_ {z_j\in {\mathcal L},  d(z_j, \Omega)<1/4}  \max_{| w-z_j|\le  1/4    } \left|f(w)e^{-\varphi(w)}\right|^p\\
    & \le & C  \sum_ {z_j\in {\mathcal L},  d(z_j, \Omega)<1/4}  \int _{| w-z_j|<  1/2   } \left|f(w)e^{-\varphi(w)}\right|^p dV(w)\\
    & \le & C  \int _{\Omega^+  } \sum_ {z_j\in {\mathcal L},  d(z_j, \Omega)<1/4} \chi_ {B(z_j,  1/2)   }(w)   \left|f(w)e^{-\varphi(w)}\right|^p dV(w)\\
    & \le & C   \int _{\Omega^+  }   \left|f(w)e^{-\varphi(w)}\right|^p dV(w).
 \eqm
It is easy to check that the constants $C$ above depend only on $p$ and $n$.
\end{proof}

With the assumption that $w \mapsto TK_w$ is conjugate holomorphic, we know $\langle T K_w, K_z \rangle$ is conjugate holomorphic with $w$. And also, $\langle T K_z, K_w \rangle_{F^2_\varphi}$ is holomorphic with $w$. For $0<p<1$, apply Lemma \ref{l:2.4} to get
\beqm
    && \int_\Omega |\langle Tk_z,  k_w\rangle_{F^2_\varphi}| dV(w) =   \int_\Omega |\langle Tk_z,  K_w\rangle_{F^2_\varphi} e^{-\varphi(w)} | dV(w) \\
    &&\verb#                  #\le C  \left( \int_{\Omega^+}|\langle Tk_z,  K_w\rangle_{F^2_\varphi} e^{-\varphi(w)} |^p dV(w) \right)^{\frac{1}{p}}\\
    &&\verb#                  # =C  \left( \int_{\Omega^+} |\langle Tk_z,  k_w\rangle_{F^2_\varphi}  |^p dV(w) \right)^{\frac{1}{p}}.
\eqm
And, similarly
$$
     \int_\Omega |\langle k_z,  Tk_w\rangle_{F^2_\varphi}| dV(w) \le C \left( \int_{\Omega^+} |\langle k_z,  Tk_w\rangle_{F^2_\varphi}  |^p dV(w) \right)^{\frac{1}{p}}.
$$
These two inequalities  tell  us  $ \textrm{WL}_p^\varphi \subset  \textrm{WL}_1^\varphi$  with $0<p\le 1$. With the relation $\langle T K_z, K_w\rangle_{F^2_\varphi}=\langle K_z, T^* K_w\rangle_{F^2_\varphi}$,  we know $T^*$ is well defined on $\mathcal D$. In \cite{IMW13} it is pointed out that $\textrm{WL}_1^\varphi$ is contained in the set of all bounded operators on $F^p_\varphi$  for all $1\le p< \infty$. When $0<p<1$, we have the two following lemmas.

\begin{lm}
\label{l:2.5}
For $0<p\leq 1$, if    $T\in \textrm {WL}_p^\varphi$  then $T$ is bounded on $F^p_\varphi$.
\end{lm}

\begin{proof}
Set
$$
       G(T)= \max\left \{ \sup_{z\in \mathbb{C}^n}  \int_{\mathbb{C}^n} \left |\langle Tk_z,
     k_w\rangle_{F^2_\varphi}\right|^p
        dV(w),   \sup_{z\in \mathbb{C}^n}  \int_{\mathbb{C}^n} \left |\langle k_z,
       Tk_w\rangle_{F^2_\varphi}\right|^p
      dV(w) \right\}.
$$
 Then,  $G(T)<\infty$.
 For $f\in \mathcal D$, we have
\begin{equation}
\label{e:2.5}
   Tf(z)= \langle Tf, K_z  \rangle_{F^2_\varphi} =  \langle f, T^*K_z\rangle_{F^2_\varphi}= \int_{\mathbb{C}^n} f(w) \langle K_w, T^*K_z  \rangle_{F^2_\varphi} e^{-2\varphi(w)} dV(w).
\end{equation}
 Applying \eqref{e:2.5}, Lemma \ref{l:2.2} (estimate \hyperlink{l:2.2e3}{(3)}) and Lemma \ref{l:2.4} with $\Omega=\mathbb{C}^n$  to have
\beqm
          && \left| Tf(z)e^{-\varphi(z)} \right| ^p  \le C   \left( \int_{\mathbb{C}^n} \left|  f(w) \langle K_w, T ^* k_z  \rangle_{F^2_\varphi}\right| e^{- 2\varphi(w)} dV(w)\right)^p\\
           &&\verb#            # \le C   \left( \int_{\mathbb{C}^n}|  f(w)   \langle T k_w,    k_z\rangle_{F^2_\varphi}  |  e^{- \varphi(w)} dV(w)\right)^p\\
           &&\verb#            # \le C   \int_{\mathbb{C}^n}\left| f(w)   \langle T k_w,    k_z\rangle_{F^2_\varphi} e^{- \varphi(w)}\right|^p  dV(w) .
\eqm
Now, integrate both sides over $\mathbb{C}^n$, and apply Fubini's Theorem to obtain
    $$
      \|T f\|^p_{p, \varphi} \le C \int_{\mathbb{C}^n} \left|  f(w)  e^{- \varphi(w)}\right|^p  dV(w) \int_{\mathbb{C}^n} \left|\langle T k_w,    k_z\rangle_{F^2_\varphi} \right|^p  dV(z)
       = C G(T) \|f\|^p_{p, \varphi}.
    $$
 When $0<p<1$, although $F^p_\varphi$ is only a    Fr\'{e}chet space,  with $P|_{F^p_\varphi}= \textrm{Id}$   we know that $ \mathcal D $ is dense in $F^p_\varphi$.
  Therefore, $T$ is bounded on $F^p_\varphi$.
\end{proof}

Isralowitz, Mitkovski and the third author demonstrated in \cite{IMW13} that  $\textrm{WL}_1^\varphi$ is a $*$-algebra.  Lemma \ref{t:2.6} tells us   $\textrm{WL}_p^\varphi$ is  closed   under the $F^p_\varphi$ operator norm while $0<p\leq 1$.


\begin{lm}
\label{t:2.6}
For $0<p\leq 1$, $\textrm {WL}_p^\varphi$ is closed under the operator norm on $F^p_\varphi$.
\end{lm}

\begin{proof}
 We only need to prove $\overline{\textrm{WL}}_p^\varphi= \textrm{WL}_p^\varphi$. For $T\in  \overline{\textrm {WL}}_p^\varphi$ we show
  $$
     \lim_{r\to \infty } \sup_{z\in \mathbb{C}^n} \int_{B(z, r)^c}  \left|\langle Tk_z, k_w \rangle_{F^2_\varphi}\right|^p dV(w) =0.
  $$
  In fact, for any $\varepsilon >0$ we have some $A_\varepsilon\in   \textrm {WL}_p^\varphi $ such that $\|T-A_\varepsilon\|_{F^p_\varphi\to F^p_\varphi} <\varepsilon$.
  For this $A_\varepsilon$ we have some $r$ such that
  $$
    \sup_{z\in \mathbb{C}^n}  \int_{B(z, r)^c}  \left|\langle A_\varepsilon k_z, k_w \rangle_{F^2_\varphi}\right|^p dV(w)<\varepsilon.
  $$
This implies
  \beqm
   &&\verb#  #  \int_{B(z, r)^c}  \left|\langle Tk_z, k_w \rangle_{F^2_\varphi}\right|^p dV(w)\\
     &&\le    \int_{B(z, r)^c}  \left|\langle( T -A_\varepsilon) k_z, k_w \rangle_{F^2_\varphi}\right|^p
            dV(w) + \int_{B(z, r)^c}  \left|\langle A_\varepsilon k_z, k_w \rangle_{F^2_\varphi}\right|^p dV(w) \\
    &&\le \int_{\mathbb{C}^n}  \left|\langle( T -A_\varepsilon) k_z, k_w \rangle_{F^2_\varphi}\right|^p dV(w) +
              \int_{B(z, r)^c}  \left|\langle A_\varepsilon k_z, k_w \rangle_{F^2_\varphi}\right|^p dV(w)\\
      &&=\left\|\left(T-A_\varepsilon\right) k_z\right\|^p_{p, \varphi} +  \int_{B(z, r)^c}  \left|\langle A_\varepsilon k_z, k_w \rangle_{F^2_\varphi}\right|^p dV(w)\\
       && \le  \left\|T-A_\varepsilon\right\|^p_{F^p_\varphi\to F^p_\varphi} \|k_z\|^p_{p, \varphi}+ \int_{B(z, r)^c}
              \left|\langle A_\varepsilon k_z, k_w \rangle_{F^2_\varphi}\right|^p dV(w)\\
        &&  <C \varepsilon,
     \eqm
 where the constant $C$ does not depend on $\varepsilon$.
 \end{proof}

To  characterize the compactness of those  $T\in \textrm{WL}^\varphi_p$  in the case $0<p\le 1$, we will borrow ideas from
 \cite{Su07} and will be approximating a given operator $T\in {\textrm{WL}}_p^\varphi $ by infinite sums of well localized pieces. To get this approximation we need the following covering lemma from \cite{IMW13}.   See also \cite{BI12}*{Lemma 3.1}.

\begin{lm}
\label{l:3.1}
There exists a positive integer $N$ such that for each $r>0$ there is a   covering ${\mathcal  F}_r= \{F_j\}_{j=1}^\infty$ of $\mathbb{C}^n$ by
disjoint Borel sets satisfying:
\begin{enumerate}[{\normalfont (1)}]
\item every point of $\mathbb{C}^n$ belongs to at most $N$ of the sets $G_j= \{z:  d(z, F_j)\le r\}$;
\item $ \textrm{diam} \,   F_j\le 2r$ for every $j$.
\end{enumerate}
\end{lm}

Notice that, if $r>1$, we have some absolute constant $N>0$ such that
\begin{equation}
\label{e:3.1}
\sum_{j=1}^\infty \chi_{F_j^+}(w)\le  \sum_{j=1}^\infty \chi_{G_j}(w)\leq \sum_{j=1}^\infty \chi_{G_j^+}(w)\le   N\quad \forall w\in \mathbb{C}^n.
\end{equation}

This covering ${\mathcal F}_r$ can also be chosen in a simple way.  For example, let  $\{a_j\}$  be an enumeration of the lattice  $\frac{2r}{\sqrt{n}}{\mathbb{Z}}^{2n}$. And take $F_j$ to be the cube  with centers $a_j$, side-length $\frac{2r}{\sqrt{n}}$  and half of the boundary so  that $\cup_{j=1}^\infty F_j ={\mathbb{C}}^n$, $F_j\cap F_k=\emptyset$ if $j\neq k$.

\begin{prop}\label{p:3.2}
Let $0< p\leq 1$ and $T\in {\textrm{WL}}_p^\varphi$. Then for every  $\varepsilon>0$,
there is some $r >0$ sufficiently  large   such that,  for the covering $\{F_j\}_{j=1}^\infty$ and $\{G_j\}_{j=1}^\infty$ (associated to $r$)  from Lemma \ref{l:3.1}, we have
\begin{equation}\label{e:3.2}
\left \|T-P\left(\sum_{j=1}^\infty M_{\chi_{F_j}}
    TPM_{\chi_{G_j}}\right)\right\|_{F^p_\varphi\to F^p_\varphi} <\varepsilon.
\end{equation}
\end{prop}

\begin{proof}
Let  $T\in \textrm {WL}_p^\varphi$ be given. For $\varepsilon>0$,    we have some  $r>0$ sufficiently large (we may assume $r>10$) such that
 \beqm
     \int_{B(z,\,  r-1)^c}  | \langle  Tk_{z}, k_w \rangle_{F^2_\varphi} |^p dV(w)<\varepsilon \  \textrm{ and } \  \int_{B(z,\,  r-1)^c}  | \langle k_{z}, Tk_w \rangle_{F^2_\varphi} |^p dV(w)<\varepsilon.
 \eqm
 Take $\{F_j\}_{j=1}^{\infty}$ and $\{G_j\}_{j=1}^{\infty}$ to be as in Lemma \ref{l:3.1} with  $r$. For $w\in F_j^+$ and $u\in G_j^c$ we  have $|u-w|>r-1$,  then $u \in B(w, r-1)^c$.
 That is $G_j^c\subset B(w, r-1)^c$ whenever $w\in F_j^+$. Hence, for $w\in F_j^+$,
  \beqm
          &&\left|  \left  ( TP M_{\chi_{G_j^c}}f \right)(w)\right|  =  \left|  \langle P M_{\chi_{ G_j ^c}}f, T^*K_w
                \rangle_{F^2_\varphi} \right|\\
                &&\verb#               # =  \left| \langle  M_{\chi_{ G_j ^c}}f, T^*K_w
                \rangle_{F^2_\varphi} \right|\\
     &&\verb#               #  \le   \int_{G_j^c }\left|f(u)\right| \left| \langle K_u, T^*K_w
                 \rangle_{F^2_\varphi}\right| e^{-2 \varphi(u)} dV(u)\\
      &&\verb#               #  \le   \int_{B(w,r-1)^c }\left|f(u)\right| \left| \langle K_u, T^*K_w
                 \rangle_{F^2_\varphi}\right| e^{-2 \varphi(u)} dV(u).
\eqm
  Set $S=
TP- \sum\limits_{j=1}^\infty M_{\chi_{F_j}}
    TPM_{\chi_{G_j}}  $.
 Then
                 \beqm
              &&\verb#   #|PSf(z)|^p\\
              &&\le \left(\int_{{\mathbb{C}}^n}\left |Sf(w)K(z,w)e^{-2\varphi(w)}\right|dV(w)\right)^p
                    \\
                  &&= \left(\int_{{\mathbb{C}}^n}\left |\sum_{j=1}^\infty  M_{\chi_{F_j}} TPM_{\chi_{G_j^{c}}} f(w)\right||K(z,w)|e^{-2\varphi(w)} dV(w)\right)^p
                  \\
                   &&= \left(\sum_{j=1}^\infty \int_{{\mathbb{C}}^n}\left | M_{\chi_{F_j}} TPM_{\chi_{G_j^{c}}} f(w)\right||K(z,w)|e^{-2\varphi(w)} dV(w)\right)^p
                  \\
                  &&\le \sum_{j=1}^\infty \left(\int_{F_j}\left | TPM_{\chi_{G_j^{c}}} f(w)\right||K(z,w)|e^{-2\varphi(w)} dV(w)\right)^p.
  \eqm
  Notice that $|K(z,w)|=|K(w,z)|$, applying Lemma \ref{l:2.4} twice to above, we get
   \beqm
              &&\verb#   #|PSf(z)|^p\\
                  &&\le C\sum_{j=1}^\infty \int_{F_j^+}\left | TPM_{\chi_{G_j^{c}}} f(w)\right|^p|K(w,z)|^pe^{-2p\varphi(w)} dV(w)
                    \\
                  && \le  C  \sum_{j=1}^\infty \int_{F_j^{+}}|K(w,z)|^pe^{-p\varphi(w)} \left(\int_{B(w, r-1)^c  }  \left| f(u)e^{-\varphi(u)} \langle  Tk_{u}, k_w \rangle_{F^2_\varphi} \right| dV(u)\right)^p dV(w)
                   \\
                  && \le  C  \sum_{j=1}^\infty \int_{F_j^{+}}|K(w,z)|^pe^{-p\varphi(w)}\left( \int_{B(w, r-2)^c  }  \left| f(u)e^{-\varphi(u)} \langle  Tk_{u}, k_w \rangle_{F^2_\varphi} \right|^p dV(u)\right) dV(w).
  \eqm
   By Fubini's Theorem and \eqref{e:3.1}, we get $\|PSf\|^p_{p,\varphi}$ is no more than
\beqm
  && \verb#   #C \sum_{j=1}^\infty \int_{\mathbb{C}^n}  \left| f(u)e^{-\varphi(u)}\right|^p\int_{F_j^{+}}\chi_{B(u, r-1)^c}(w)
   \left|\langle  Tk_{u}, k_w \rangle_{F^2_\varphi} \right|^p e^{-p\varphi(w)}
   \\
     &&\verb#   # \times \int_{\mathbb{C}^n} |K(w,z)|^pe^{-p\varphi(z)} dV(z)dV(w)dV(u)
   \\
     && \le CN  \int_{\mathbb{C}^n}  \left| f(u)e^{-\varphi(u)}\right|^p\left(\int_{B(u, r-1)^c}
   \left|\langle  Tk_{u}, k_w \rangle_{F^2_\varphi} \right|^p dV(w)\right)dV(u)
   \\
     && \le C \varepsilon \|f\|^p_{p, \varphi}.
  \eqm
The constants $C$ above are independent of $\varepsilon$. Notice that  $PTP=T$ on $F^p_\varphi$, so $PS= T-P\left(\sum\limits_{j=1}^\infty M_{\chi_{F_j}}
    TPM_{\chi_{G_j}}\right)$ 
    is well defined on $F^p_\varphi$ and
 the estimate \eqref{e:3.2} is proved under the restriction that $ T\in  \textrm{WL}_p^\varphi$.
 \end{proof}

\begin{lm}
\label{l:3.3}
Given  $0<p\leq 1$, there is  some constant $C$ such that for  all   bounded  linear operator $T$ on $F^p_\varphi$ and $\{F_j\}_{j=1}^{\infty}$,  $\{G_j\}_{j=1}^{\infty}$ associated to $r>1$  as in Lemma \ref{l:3.1} and each positive integer $m$, we have
 \begin{equation}
 \label{e:3.3}
     \limsup_{m\to \infty} \|P T_m\|_{F^p_\varphi\to F^p_\varphi} \le C \limsup_{m\to \infty} \sup_{w\in  \cup_{j>m} G^+_j}
    \|Tk_w \|_{p, \varphi},
 \end{equation}
where  $T_m= \sum\limits_{j>m}M_{\chi_{F_j}} T PM_{\chi_{G_j}}$.
\end{lm}

\begin{proof}
First, we are going to show
\begin{equation}
\label{e:3.4}
       \sup_{f\in F^p_\varphi\backslash \{0\} } \left\| T P \left(      \fr {\chi_{G_j} f}{\|\chi_{G_j^+}  f \|_{p, \varphi}}
          \right )\right \| _{p, \varphi}  \le C  \sup_{w\in G_j^+}
                \|T k_w\|_{p, \varphi}.
\end{equation}
In fact, given $f\in F^p_\varphi$   not identically zero, set
$$
    g_j =  P \left(      \fr {\chi_{G_j} f}{\|\chi_{G_j^+}  f \|_{p, \varphi}} \right ).
$$
  Then
$$
    g_j(z)=\int_{G_j} \fr {f(w) K(z,w) e^{-2\varphi(w)} }
              { {\|\chi_{G_j^+}f
        \|_{p, \varphi} }}   dV(w).
$$
It is trivial to see that $g_j\in F^p_\varphi$ because of the compactness of $\overline{G_j}$.
Since $T$ is bounded on $F^p_\varphi$, then
$$
\left|T  (g_j)(z)\right|\leq \int_{G_j} \fr {|f (w)| |TK_w(z)| e^{-2\varphi(w)} }
              { {\|\chi_{G_j^+}f
        \|_{p, \varphi} }}   dV(w)
$$
Note that $T K_w$ is conjugate holomorphic respecting to $w$. From  Lemma \ref{l:2.4} we have
\beqm
    \verb#     #
    && \|T  (g_j)\|_{p, \varphi}^p
        \le    \int_{\mathbb{C}^n} \left(  \int_{G_j}\fr {|f (w)| \left|\overline{TK_w(z)}\right| e^{-2\varphi(w)}  } {\|\chi_{G_j^+}f
        \|_{p, \varphi} } dV(w)  \right )^p e^{-p \varphi(z)}  dV(z) \\
        &&  \verb#        #   \le C  \int_{\mathbb{C}^n}  \left(  \int_{G_j^+}\fr {\left|f(w )\right|^p\left|\overline{TK_w(z)}\right| ^p  e^{-2p\varphi(w)} } {\|\chi_{G_j^+}f
        \|_{p, \varphi}^p }  dV(w) \right)    e^{-p \varphi(z)}  dV(z) \\
    &&\verb#        # \le C \int_{G_j^+} \fr {\left|f(w)e^{-\varphi(w)}\right|^p   } {\|\chi_{G_j^+}f
        \|_{p, \varphi}^p  }  \left( \int_{\mathbb{C}^n} \left| T k_w(z) e^{-\varphi
        (z)}\right|^p
        dV(z)\right) dV(w)\\
     &&\verb#        # \le C  \sup_{w\in G_j^+}
                \|T k_w\|^p_{p, \varphi}  \int_{G_j^+} \fr {\left|f(w)e^{-\varphi(w)}\right|^p   dV(w)  } {\|\chi_{G_j^+}f
        \|_{p, \varphi}^p }  \\
     &&\verb#        #  = C   \sup_{w\in G_j^+}
                \|T k_w\|^p_{p, \varphi}.
                 \eqm
This gives \eqref{e:3.4}. To prove \eqref{e:3.3}, we have  from Lemma \ref{l:2.4} that
\beqm
   &&\verb#  # \left| P \left( \chi_{F_j}(\cdot) \int_{G_j}\fr { \langle f, k_w  \rangle_{F^2_\varphi}  (T k_w)(\cdot)  } {\|\chi_{G_j^+}f
        \|_{p, \varphi} } dV(w)  \right )(z) \right|^p\\
   &&\le \left| \int_{F_j}    K(z, u)  e^{- 2\varphi(u)} \int_{G_j}\fr { \langle f, k_w  \rangle_{F^2_\varphi}   (T k_w)(u) } {\|\chi_{G_j^+}f
        \|_{p, \varphi} } dV(w) dV(u) \right|^p   \\
   && \le C  \int _{F_j^+} |K(z, u)|^p  e^{- 2p\varphi(u)} \left|\int_{G_j}\fr { \langle f, k_w  \rangle_{F^2_\varphi}  (T k_w)(u) } {\|\chi_{G_j^+}f
        \|_{p, \varphi}^p } dV(w)\right|^p dV(u)\\
   && \le C  \int _{F_j^+} |K(z, u)|^p  e^{- 2p\varphi(u)} \left(\int_{G_j^+}\fr { |\langle f, k_w  \rangle_{F^2_\varphi} |^p | (T k_w)(u)|^p  } {\|\chi_{G_j^+}f
        \|_{p, \varphi}^p } dV(w)\right) dV(u).
   \eqm
Hence, integrating both sides and interchanging the order of integrations we obtain
\beqm
 &&  \verb#   # \left  \|P\left(M_{\chi_{F_j}} T g_j\right) \right \|_{p, \varphi}^p\\
  && \le C \int_{G_j^+}\fr { |\langle f, k_w  \rangle_{F^2_\varphi} |^p   } {\|\chi_{G_j^+}f
        \|_{p, \varphi}^p }\int _{F_j^+} | (T k_w)(u)|^p e^{- 2p\varphi(u)}\int_{\mathbb{C}^n} |K(z, u)|^p e^{-p\varphi(z)} dV(z) dV(u) dV(w)\\
  && \le C  \int_{G_j^+}\fr { |\langle f, k_w  \rangle_{F^2_\varphi} |^p   } {\|\chi_{G_j^+}f
        \|_{p, \varphi}^p }\left(\int _{F_j^+} | (T k_w)(u)|^p e^{-  p\varphi(u)} dV(u)\right)dV(w).
\eqm
This gives
\beqm
 \left  \|P\left(M_{\chi_{F_j}} T g_j\right) \right \|_{p, \varphi}^p
   \le C \left(\sup_{w\in G_j^+} \|Tk_w\|_{p, \varphi}^p\right)   \int_{G_j^+}\fr { |\langle f, k_w  \rangle_{F^2_\varphi} |^p   } {\|\chi_{G_j^+}f
        \|_{p, \varphi}^p } dV(w) =C  \sup_{w\in G_j^+} \|Tk_w\|_{p, \varphi}^p.
\eqm
Therefore, \eqref{e:3.1} yields
\beqm
    &&  \|P T_m f\|^p_{p, \varphi } \le \sum_{j>m} \|P M_{\chi_{F_j}} T PM_{\chi_{G_j}}f \|_{p, \varphi}^p   \\
    && \verb#         #  =\sum_{j>m} \| P\left(M_{\chi_{F_j}} T g_j\right)\|_{p, \varphi}^p \,  \|\chi_{G_j^+}f\|_{p, \varphi}^p \\
    && \verb#         # \le C    \sum_{j>m} \sup_{w\in G_j^+} \|Tk_w\|_{p, \varphi}^p \|\chi_{G_j^+}f\|^p_{p, \varphi}\\
    && \verb#         # \le C N\left( \sup_{w\in \cup_{j>m} G_j^+} \|Tk_w\|_{p, \varphi}^p \right)  \|f\|^p_{p, \varphi}.
\eqm
 From this, \eqref{e:3.3} follows.
\end{proof}


In the case of $1\leq p< \infty$,  the projection $P$ is bounded from $L^p_\varphi$ to $F^p_\varphi$, and so is $P M_{\chi_{E}}$ when $E\subset \mathbb{C}^n$ is measurable. But $P$ is not bounded on $L^p_\varphi$  if $0<p<1$. The following lemma,  Lemma \ref{l:3.4},  says   $P M_{\chi_{E}}$ is still bounded  on  $F^p_\varphi$.

\begin{lm}
\label{l:3.4}
Suppose $0<p\leq 1$. There exists some constant $C$ such that for any domain $E$ in $\mathbb{C}^n$ we have
$
\|P M_{\chi_{E}}\|_{F^p_\varphi\rightarrow F^p_\varphi}\leq C
$.
\end{lm}

\begin{proof}
Suppose $E\in \mathbb{C}^n$ is a domain.   For $f\in F^p_\varphi$, we have $|f(w)K(z,w)|=|f(w)K(w,z)|$. Lemma \ref{l:2.4} and Lemma \ref{l:2.2}, estimate \hyperlink{l:2.2e3}{(3)} gives
 \beqm
&&  \left\|PM_{\chi_E}f\right\|_{p,\varphi}^p=\int_{\mathbb{C}^n}\left|\int_{E}f(w)K(z,w)e^{-2\varphi(w)}dV(w)\right|^pe^{-p\varphi(z)}dV(z)
\\
&&\verb#           #\leq C\int_{\mathbb{C}^n}\left(\int_{E^+}\left|f(w)K(w,z)e^{-2\varphi(w)}\right|^p dV(w)\right)e^{-p\varphi(z)}dV(z)
\\
&&\verb#           #= C\int_{E^+}\left|f(w)\right|^p e^{-2p\varphi(w)}\left(\int_{\mathbb{C}^n}\left|K(w,z)e^{-\varphi(z)}\right|^p dV(z)\right)dV(w).
\\
&&\verb#           #\leq C\|f\|_{p,\varphi}^{p}.
\eqm
\end{proof}


\begin{lm}
\label{l:3.5}
Suppose $0<p\leq 1$ and  $T\in {\mathcal K}(F^p_\varphi)$. Then
$$
 \lim\limits_{R\rightarrow \infty}\|P M_{\chi_{B(0,R)}} T-T\|_{F^p_\varphi\rightarrow F^p_\varphi}=0.
$$
\end{lm}

\begin{proof}
Notice that, $PT=T$ on $F^p_\varphi$.  For $f\in F^p_\varphi$ with  $\|f\|_{p,\varphi}\leq 1$, we get
 \beqm
&&  \left\|\left(PM_{\chi_{B(0,R)}}T-T\right)f\right\|_{p,\varphi}^p=\left\|\left(PM_{\chi_{B(0,R)}}T-PT\right)f\right\|_{p,\varphi}^p
\\
&&\verb#                     #=\int_{\mathbb{C}^n}\left|\int_{|w|\geq R}Tf(w)K(z,w)e^{-2\varphi(w)}dV(w)\right|^pe^{-p\varphi(z)}dV(z).
\eqm
Then by Lemma \ref{l:2.4},
 \beqm
&&  \left\|\left(PM_{\chi_{B(0,R)}}T-T\right)f\right\|_{p,\varphi}^p\leq C\int_{\mathbb{C}^n}\left(\int_{|w|\ge R-1}\left|Tf(w)K(w,z)e^{-2\varphi(w)}\right|^pdV(w)\right)e^{-p\varphi(z)}dV(z)
\\
&&\verb#                     #=\int_{|w|\ge R-1}\left|Tf(w)e^{-2\varphi(w)}\right|^p \left(\int_{\mathbb{C}^n}\left|K(w,z)\right|^pe^{-p\varphi(z)}dV(z) \right)dV(w) \\
&&\verb#                     #\leq C\int_{|w|\ge R-1}\left|Tf(w)e^{-\varphi(w)}\right|^p dV(w).
\eqm
Since $T\in {\mathcal K}(F^p_\varphi)$, $\left\{Tf: \ f\in F^p_\varphi \textrm{ with  } \|f\|_{p,\varphi}\leq 1\right\}\subset F^p_\varphi$ is relatively compact.  By \cite{HL14}*{Lemma 3.2}, for  each $\varepsilon >0$ there is some $R>0$ such that
$$
\sup\limits_{ f\in F^p_\varphi, \|f\|_{p,\varphi}\leq 1}\int_{|w|>R-1}\left|Tf(w)e^{-\varphi(w)}\right|^p dV(w)<\varepsilon^p.
$$
Therefore,
$$
\|P M_{\chi_{B(0,R)}} T-T\|_{F^p_\varphi\rightarrow F^p_\varphi}=\sup\limits_{ f\in F^p_\varphi, \|f\|_{p,\varphi}\leq 1} \left\|\left(PM_{\chi_{B(0,R)}}T-T\right)f\right\|_{p,\varphi}<C\varepsilon,
$$
where $C$ is independent of $\varepsilon$.
\end{proof}

\begin{lm}
\label{l:3.6}
Suppose $0<p\leq 1$. Then for $T$   bounded on $F^p_\varphi$ we have
$$
\|T\|_{e, F^p_\varphi}\simeq \limsup\limits_{R\rightarrow\infty}\|PM_{\chi_{B(0,R)^c}}T\|_{F^p_\varphi\rightarrow F^p_\varphi}.
$$
\end{lm}

\begin{proof}
For any $R>0$, $PM_{\chi_{B(0,R)}}$  is a Toeplitz operator induced by $\chi_{B(0,R)}$, Lemma \ref{l:3.3} from \cite{HL14} tells us   it is compact on $F^p_\varphi$. Given
$T$   bounded on $F^p_\varphi$,    $PM_{\chi_{B(0,R)}}T$ is compact. Thus,
$$
\|T\|_{e, F^p_\varphi}\leq\|T-PM_{\chi_{B(0,R)}}T\|_{F^p_\varphi\rightarrow F^p_\varphi}.
$$
This yields
 $$
\|T\|_{e, F^p_\varphi}\leq  \limsup\limits_{R\rightarrow\infty}\|PM_{\chi_{B(0,R)^c}}T\|_{F^p_\varphi\rightarrow F^p_\varphi}.
$$
On the other hand, for any $A\in {\mathcal K}(F^p_\varphi)$, Lemma \ref{l:3.5} shows
$$
\limsup\limits_{R\rightarrow\infty}\|PM_{\chi_{B(0,R)^c}}A\|_{F^p_\varphi\rightarrow F^p_\varphi}=0.
$$
From Lemma \ref{l:3.4}, we know
\beqm
&&\limsup\limits_{R\rightarrow\infty}\|PM_{\chi_{B(0,R)^c}}T\|_{F^p_\varphi\rightarrow F^p_\varphi}=\limsup\limits_{R\rightarrow\infty}\|PM_{\chi_{B(0,R)^c}}(T-A)\|_{F^p_\varphi\rightarrow F^p_\varphi}
\\
&&\verb#                       #\leq \limsup\limits_{R\rightarrow\infty}\|PM_{\chi_{B(0,R)^c}}\|_{F^p_\varphi\rightarrow F^p_\varphi}
\|T-A\|_{F^p_\varphi\rightarrow F^p_\varphi}
\\
&&\verb#                       #\leq C\|T-A\|_{F^p_\varphi\rightarrow F^p_\varphi}.
\eqm
Hence,
$$
\limsup\limits_{R\rightarrow\infty}\|PM_{\chi_{B(0,R)^c}}T\|_{F^p_\varphi\rightarrow F^p_\varphi}\leq C\|T\|_{e, F^p_\varphi}.
$$
\end{proof}

Now we are in the position to characterize those compact operators in $\textrm{WL}_p^\varphi$ with $0<p\leq 1$, which extends the main results in \cites{Is13, IMW13, XZ13} to the small exponential case.

\begin{thm}
\label{t:3.7}
Let $0<p\leq 1$ and $T\in \textrm{WL}_p^\varphi$. The following  statements are equivalent:
\begin{enumerate}[{\normalfont (A)}]
\item $T\in {\mathcal K}(F^p_\varphi)$;

\item  $ \lim\limits_{z\to \infty} \sup\limits_{w\in B(z, r)}
   \left|\langle Tk_z, k_w \rangle_{F^2_\varphi}\right|= 0$ for any    $r>0$;

\item   $ \lim\limits_{z\to \infty} \sup\limits_{w\in \mathbb{C}^n}
   \left|\langle Tk_z, k_w \rangle_{F^2_\varphi}\right|= 0$;

\item  $ \lim\limits_{z\to \infty} \|T k_z\|_{p, \varphi} =0$.
\end{enumerate}
\end{thm}

\begin{proof}
It is trivial that (C)$\Rightarrow$(B). We will    show  the implication  (B)$\Rightarrow$(D)   under the hypothesis $T\in  {\textrm{WL}}_p^\varphi$.  In fact, for any $\varepsilon>0$, by  \eqref{e:2.2} we have some  $r>0$ such that
 $$
 \sup\limits_{z\in \mathbb{C}^n}\int_{B(z, r)^c} \left|\langle Tk_z, k_w \rangle_{F^2_\varphi}\right|^p  dV(w)<\varepsilon.
 $$
  Combining the above inequality
with (B), we get
  \beqm
      && \|Tk_z\|_{p, \varphi}^p  =\int_{\mathbb{C}^n}   \left|\langle Tk_z, k_w \rangle_{F^2_\varphi}\right|^p  dV(w) \\
      && \verb#       #= \left(\int_{B(z, r)^c}  + \int_{ B(z, r )}  \right)\left|\langle Tk_z, k_w \rangle_{F^2_\varphi}\right|^p  dV(w) \\
      && \verb#       #\le  \varepsilon +  A(B(z, r)) \sup_{w\in  B(z, r)}  \left|\langle Tk_z, k_w \rangle_{F^2_\varphi}\right|^p  \\
       && \verb#       #\le  \varepsilon +  Cr^{2n}\left( \sup_{w\in  B(z, r)}  \left|\langle Tk_z, k_w \rangle_{F^2_\varphi}\right|\right)^p \\
        && \verb#       # < 2 \varepsilon
  \eqm
whenever $|z|$ is sufficiently large.  Therefore,  (B) implies (D).

Suppose $T$ satisfies (D). By  Lemma \ref{l:2.3}  we know
\begin{equation}
\label{e:3.5}
  \left|\langle Tk_z, k_w \rangle_{F^2_\varphi}\right| =\left|Tk_z(w) e^{-\varphi(w)}\right| \le  C\left(\int_{B(w, 1)} \left|Tk_z(u) e ^{-\varphi(u)}\right|^p dV(u)\right)^{\fr 1p}\le C  \|Tk_z\|_{p, \varphi}.
\end{equation}
Then,
$$
    \sup _{w\in \mathbb{C}^n}
   \left|\langle Tk_z, k_w \rangle_{F^2_\varphi}\right| \le C   \|Tk_z\|_{p, \varphi}
$$
which gives the implication  (D)$\Rightarrow$(C).

To prove (D)$\Rightarrow$(A),  
given  $\varepsilon>0$  we  pick some $r>10 $ with sets $\{F_j\}_j$ and $\{G_j \}_j$
as in Proposition \ref{p:3.2} so that
$$
   \left\|T -P\left(\sum_{j=1}^\infty M_{\chi_{F_j}} T PM_{\chi_{G_j}}\right)\right \|_{F^p_\varphi\to F^p_\varphi}
    <\varepsilon.
$$
For each positive integer $m$, set $T_m= \sum\limits_{j>m}M_{\chi_{F_j}} T PM_{\chi_{G_j}}$. Since $ P\left(\sum\limits_{j=1}^m M_{\chi_{F_j}} T PM_{\chi_{G_j}}\right)$ is compact on $F^p_\varphi$, we get
 \begin{equation}
 \label{e:3.6}
      \|T\|^p_{e, F^p_\varphi} \le  \left\| T - P\left(\sum_{j=1}^m M_{\chi_{F_j}} T PM_{\chi_{G_j}}\right) \right\|^p_{F^p_\varphi\to
       F^p_\varphi} < \varepsilon^p+ \|PT_m\|^p_{F^p_\varphi\to F^p_\varphi}.
 \end{equation}
Suppose  $T$ satisfies  (D), then there exists $t>0$ such that $\|T k_z\|_{p, \varphi}<\varepsilon$ for $|z|\geq t$.
Notice that, $ \cup_{j>m} G^+_j\subset B(0, t)^c $ whenever $m$ is large enough.  So,
\eqref{e:3.3} in Lemma \ref{l:3.3} and \eqref{e:3.6}  imply
    $\|T\|_{e, F^p_\varphi} =0$ which gives the compactness of $T$.

 To finish our proof, we only need to  prove the implication (A)$\Rightarrow$(B).   Given $T\in {\mathcal K}(F^p_\varphi)$, Lemma \ref{l:3.5} tells us
\begin{equation}
\label{e:3.7}
 \lim\limits_{R\rightarrow \infty}\|P M_{\chi_{B(0,R)}} T-T\|_{F^p_\varphi\rightarrow F^p_\varphi}=0.
\end{equation}
First,  we claim that
\begin{equation}
\label{e:3.8}
\lim\limits_{z\to \infty} \sup\limits_{w\in B(z, r)}
   \left|\langle P M_{\chi_{B(0,R)}}Tk_z, k_w \rangle_{F^2_\varphi}\right|= 0.
\end{equation}
  In fact, Lemma \ref{l:3.4} shows $P M_{\chi_{B(0,R)}}Tk_z\in F^p_\varphi\subset F^2_\varphi$, we obtain
  \beqm
    &&  \left|\langle P M_{\chi_{B(0,R)}}Tk_z, k_w \rangle_{F^2_\varphi}\right|= \left|\langle  M_{\chi_{B(0,R)}}Tk_z, k_w \rangle_{F^2_\varphi}\right|\\
     &&\verb#                   # \leq \int_{B(0,R)}  \left| Tk_z(u) \overline{k_w(u)}\right|e^{-2\varphi(u)}dV(u) \\
    &&\verb#                   # \le \|Tk_z\|_{\infty, \varphi} \int_{B(0,R)}  \left| k_w(u)\right|e^{-\varphi(u)}dV(u)\\
      &&\verb#                   #  \leq C \|Tk_z\|_{p, \varphi}  \sup\limits_{u\in B(0,R)}  \left| k_w(u)\right|e^{-\varphi(u)}\\
       && \verb#                   #  \leq C\|T\|_{F^p_\varphi\rightarrow F^p_\varphi}\|k_z\|_{p, \varphi} e^{-\theta|w|}\\
        && \verb#                   # \leq C e^{-\theta|w|},
     \eqm
where the constants $C$ are independent of $z$ and $w$. Hence, \eqref{e:3.8} is true. Using \hyperlink{l:2.2e3}{(3)} in Lemma \ref{l:2.2}  and \eqref{e:3.7}  to get that
 \beqm
    && \left|\left\langle \left(T-P M_{\chi_{B(0,R)}}T\right)k_z, k_w\right \rangle_{F^2_\varphi}\right|\leq\left\| \left(T-P M_{\chi_{B(0,R)}}T\right)k_z\right\|_{\infty, \varphi}\|k_w\|_{1, \varphi}\\
     &&\verb#                          # \leq C \left\| \left(T-P M_{\chi_{B(0,R)}}T\right)k_z\right\|_{p, \varphi}
     \\
    &&\verb#                          #  \le C\|T-P M_{\chi_{B(0,R)}}T \|_{F^p_\varphi\rightarrow F^p_\varphi}\|k_z\|_{p, \varphi}
    \\
    &&\verb#                          #  \le C\|T-P M_{\chi_{B(0,R)}}T \|_{F^p_\varphi\rightarrow F^p_\varphi}  \rightarrow0
     \eqm
as $R\rightarrow\infty$. Combining the above
with \eqref{e:3.8}, we obtain
\beqm
      \sup\limits_{w\in B(z, r)}\left|\langle Tk_z, k_w \rangle_{F^2_\varphi}\right|\leq \sup\limits_{w\in B(z, r)}\left|\left\langle P M_{\chi_{B(0,R)}}Tk_z, k_w\right \rangle_{F^2_\varphi}\right|+ \sup\limits_{w\in B(z, r)}\left|\left\langle \left(T-P M_{\chi_{B(0,R)}}T\right)k_z, k_w\right \rangle_{F^2_\varphi}\right|
    \eqm
goes to $0$ as $z\rightarrow\infty$.
\end{proof}

If $1<p<\infty$, $k_z\to 0$ weakly   on $F^p_\varphi$, which  implies $\lim\limits_{z\to \infty} \|T(k_z)\|_{p, \varphi} =0$ for $T\in {\mathcal K}(F^p_\varphi)$. Theorem 1.2 in \cite{IMW13} tells us that the equivalence from (A) to (D) remains true for $T\in \textrm{WL}_p^\varphi$ if $1<p<\infty$. For  our later applications, we exhibit the following
result.

\begin{thm}
\label{t:3.8}
Let $0<p< \infty$ and $T\in  \textrm{WL}_p^\varphi$. The following  statements are equivalent:
\begin{enumerate}[{\normalfont (A)}]

\item  $T\in {\mathcal K}(F^p_\varphi)$;

\item  $ \lim\limits_{z\to \infty} \sup\limits_{w\in B(z, r)}\left|\langle Tk_z, k_w \rangle_{F^2_\varphi}\right|= 0$ for any $r>0$;

\item  $ \lim\limits_{z\to \infty} \sup\limits_{w\in \mathbb{C}^n}\left|\langle Tk_z, k_w \rangle_{F^2_\varphi}\right|= 0$;

\item $ \lim\limits_{z\to \infty} \|T k_z\|_{p, \varphi} =0$.
\end{enumerate}
\end{thm}


From Theorem \ref{t:3.7}, it is natural to ask whether the essential norm of $T\in\textrm{WL}_p^\varphi$  can be dominated by its behavior on  normalized reproducing kernel $k_z$? This problem has attracted much interest, see \cite{Is13}*{Section 6} for example. Our Theorem \ref{t:3.9} says the answer is affirmative when $0<p\le 1$.

\begin{thm}
\label{t:3.9}
Suppose  $0<p\le 1$.  Then for  $T\in {\textrm{WL}}_p^\varphi$ we have
\begin{equation}
\label{e:3.9}
     \|T\|_{e, F^p_\varphi}
    \simeq \limsup_{z \to \infty} \|Tk_z\|_{p, \varphi}.
\end{equation}
\end{thm}

\begin{proof}
Suppose $T\in  {\textrm{WL}}_p^\varphi$.  From Lemma \ref{l:2.5} we know $T$ is bounded on $F^p_\varphi$ which implies  $\limsup\limits_{z\to \infty}
 \|Tk_z\|_{p, \varphi}<\infty$. By Theorem \ref{t:3.7}, $\|T\|_{e, F^p_\varphi}=0$   if $ \limsup\limits_{z \to \infty  } \|Tk_z\|_{p, \varphi}=0$. So, we  may assume
  $ \limsup\limits_{z \to \infty  } \|Tk_z\|_{p, \varphi}=\varepsilon_1 >0$.
    From Proposition \ref{p:3.2}, we have two sequences of sets
 $\{F_j\}_j$ and $\{G_j\}_j$ so that
 $$
     \left\| T-P\left(\sum_{j=1}^\infty M_{\chi_{F_j}} TP M_{\chi_{G_j}}\right) \right\|_{F^p_\varphi\to F^p_\varphi} <\varepsilon_1.
 $$
 Then, for $m=1, 2, \cdots$,  from \eqref{e:3.3} and \eqref{e:3.6} we have
\beqm
   \|T\|_{e, F^p_\varphi}
     \le \varepsilon_1 + \left\| PT_m \right\|_{F^p_\varphi\to F^p_\varphi}\le  \varepsilon_1 + C\sup_{z \in\cup_{j>m} G_j^+} \|Tk_z\|_{p, \varphi}.
   \eqm
Since $0<p\le 1$, Lemma \ref{l:3.3} tells us  that the constants $C$  above do not depend on  the precise choice of $\{F_j\}_j$ and $\{G_j\}_j$, and hence do not depend on $T$.   Let $m\to \infty$,  we have the desired estimate
$$
       \|T\|_{e, F^p_\varphi}^p
     \le \varepsilon_1^p +   C\limsup_{z\to \infty}  \|Tk_z\|^p_{p, \varphi}= C \limsup_{z\to \infty}  \|Tk_z\|^p_{p, \varphi}.
$$
On the other hand, fixed $R>0$, notice that $P M_{\chi_{B(0,R)}}$ is a Toeplitz operator induced by $\chi_{B(0,R)}$, which is a bounded function. So $P M_{\chi_{B(0,R)}}\in \textrm{WL}_p^\varphi$ and $P M_{\chi_{B(0,R)}}T\in \textrm{WL}_p^\varphi$, because $\textrm{WL}_p^\varphi$ is a algebra. Since $P M_{\chi_{B(0,R)}}$ is compact  and $T$ is bounded on $F^p_\varphi$ (see Lemma \ref{l:2.5}),  we get that $P M_{\chi_{B(0,R)}}T$ is compact on  $F^p_\varphi$.  Theorem \ref{t:3.7} tells us
 \begin{equation}
 \label{e:3.10}
   \lim\limits_{z\to \infty} \left\|P M_{\chi_{B(0,R)}}T k_z\right\|_{p, \varphi} =0.
 \end{equation}
  Therefore, Lemma \ref{l:3.6}, \eqref{e:3.10} and the fact that $PT=T$ yield
\beqm
&&\|T\|^p_{e, F^p_\varphi}\simeq \limsup\limits_{R\rightarrow\infty}\|PM_{\chi_{B(0,R)^c}}T\|^p_{F^p_\varphi\rightarrow F^p_\varphi}
\\&&\verb#      #\geq \limsup\limits_{R\rightarrow\infty}\limsup\limits_{z\to \infty} \left\|P M_{\chi_{B(0,R)^c}}T k_z\right\|^p_{p, \varphi}
\\
&&\verb#      #\geq\limsup\limits_{R\rightarrow\infty}\limsup\limits_{z\to \infty}\left(\|T k_z\|^p_{p, \varphi}-\left\|PM_{\chi_{B(0,R)}}T k_z\right\|^p_{p, \varphi}\right)
\\
&&\verb#      #\geq \limsup\limits_{z\to \infty}\|T k_z\|^p_{p, \varphi}.
\eqm
\end{proof}

\section{Toeplitz Operators with $\textrm{BMO}$ Symbols}
\label{s:three}

In this section,  we are going to discuss the characterizations on Toeplitz operators with $\textrm{BMO}$ symbols.
  First, we will characterize   the boundedness (and the   compactness) of Toeplitz operators $T_f$ on $F^p_{\varphi}$ with $\textrm{BMO}$ symbols   $f $.  Furthermore, we will  characterize  those compact operators on $F^p_{\varphi}$ which are in   the algebra generated by bounded Toeplitz operators with $\textrm{BMO}$ symbols. For this purpose, we  need  some more  auxiliary function spaces.

Fixed   $r>0$, recall that $B(\cdot,r)=\left\{w\in \mathbb{C}^n: |w-\cdot|<r\right\}$. Given a   locally Lebesgue integrable function $f$ on $\mathbb{C}^n$ (written as $f\in L^1_{loc}(\mathbb{C}^n)$), write
$$
\omega_{r}(f)(\cdot)=\sup\left\{|f(w)-f(\cdot)|:  w\in B(\cdot,r)\right\}
$$
and
$$
MO_{r}(f)(\cdot)=\frac{1}{V\left(B(\cdot,r)\right)}\int_{B(\cdot,r)}\left|f-\widehat{f}_{r}(\cdot)\right|dV $$
where
 $$
\widehat{f}_{r}(\cdot)=\frac{1}{V\left(B(\cdot,r)\right)}\int_{B(\cdot,r)}fdV.
$$
For $f$ on $\mathbb{C}^n$ with $f(\cdot)|k_{z}(\cdot)|^2\in L^{1}_\varphi$ for all $z\in\mathbb{C}^n$, the   Berezin transform of $f$  is defined as
$$
\widetilde{f}(z)= \int_{\mathbb{C}^n}f(w)\left|k_{z}(w)\right|^2e^
{-2\varphi(w)}dV(w).
$$
Let   $\textrm{BO}_{r}$ be the collection of all  continuous functions $f$ on $\mathbb{C}^n$ such that
$\omega_{r}(f)$
is bounded.
We use $\textrm{BA}_r$ and $\textrm{BMO}_r$ to denote respectively  the set of all  $f\in L^1_{loc}(\mathbb{C}^n)$ such that
$\widehat{|f|}_{r}$
and  $MO_{r}(f)$ are bounded on $\mathbb{C}^n$.
 The space $\textrm{BMO}$ is the family of all measurable function $f$ on $\mathbb{C}^n$ satisfying $f(\cdot)|k_{z}(\cdot)|^2\in L^{1}_\varphi$ for  $z\in\mathbb{C}^n$ and
$$
\|f\|_{\textrm{BMO}}=\sup\limits_{z\in\mathbb{C}^n}\int_{\mathbb{C}^n}\left|f(w)-\widetilde{f}(z)\right| |k_z(w)|^2e^{-2\varphi(w)}dV(w)<\infty.
$$

By Lemma 3.33 in \cite{Zh12}, we obtain that the spaces  $\textrm{BO}_{r}$ and $\textrm{BA}_r$  are independent of $r$,   they will be denoted as $\textrm{BO}$ and $\textrm{BA}$ below. The next lemma says $\textrm{BMO}_r$ is  independent of $r$ as well. \newline

\begin{lm}
\label{l:4.1}
Suppose $f\in L^1_{loc}(\mathbb{C}^n)$. The following three  statements are equivalent:
\begin{enumerate}[{\normalfont (A)}]
\item $f\in\textrm{BMO}_{r}$ for some (or any) $r>0$;

\item $f\in\textrm{BMO}$;

\item $f=f_1+f_2$, where $f_1\in \textrm{BA}$ and $f_2\in \textrm{BO}$.
\end{enumerate}
\end{lm}

\begin{proof}
For $n=1$ and $\varphi(z)= \fr \alpha 2 |z|^2$, this is Theorem 3.34 from \cite{Zh12}. For general $n$ and $\varphi$ satisfying $ dd^c \varphi \simeq \omega_0$, the proof can be carried out as that of \cite{Zh12} with a little modification. The details will be omitted here.  \end{proof}

  For $f\in \textrm{BMO}$, say $f=f_1+f_2$ with $f_1\in \textrm{BA}$ and $f_2\in \textrm{BO}$,  similar to \cite{HL11}*{Lemma 4.1}, we know the Toeplitz operator $T_{f_1}$ is well defined on $F^p_{\varphi}$. From \cite{Zh12},  $\left|f_2(z)\right|\leq a|z|+b$ with constants $a,b>0$, $T_{f_2}$
is also well defined on $F^p_{\varphi}$. Thus, $T_f$ is well defined on $F^p_{\varphi}$, where $0<p< \infty$.
 Moreover,
 \begin{equation}
 \label{e:4.2}
 \langle T_fk_z, k_w\rangle=\int_{\mathbb{C}^n}k_z(u)\overline{k_w(u)}f(u)e^{-2\varphi(u)}dV(u).
\end{equation}

   Coburn, Isralowitz and Li \cite{CIL11} proved that $T_f$  ($f\in \textrm{BMO}$)
is compact on the classical Fock space $F^2_{1/2}$
if and only if the Berezin transform  $\widetilde{f}$ vanished at the
infinity. The first two authors extended this result to the setting of $F^p_{1/2}$ with $0<p<\infty$ in \cite{HL14}. Under
the assumption that $S$ is a linear combination of operators of form $T_{f_1}\cdots T_{f_m}$ with each function $f_j$ satisfying  $\widetilde{|f_j|}$   bounded,   Isralowitz proved that $S$ is compact on $F^2_{1/2}$ if and only if  $\widetilde{S}$ vanishes at the infinity, see \cite{Is11} for details.
 In all these    references, the Weyl unitary operators acting on $F^2_\alpha$ by  $W_zf(\cdot)=k_z f(\cdot-z)$ and the involutive unitary operators $U_zf(\cdot)=k_z f(z-\cdot)$
 play as a very crucial role. Unfortunately, there are not these kinds of unitary operators on our generalized Fock space $F^p_\varphi$.

  We will use ${\mathcal B}_p^{\varphi}$ to denote the collection of  all linear combination of  the form $T_{f_1} T_{f_2}\cdots T_{f_m}$, where  each  function $f_j\in \textrm{BMO}$ and $\widetilde{f_j}$ is bounded on $\mathbb{C}^n$. \newline

\begin{thm}
\label{t:4.2}
Let $0<p<\infty$.  \vspace{2mm}

\begin{enumerate}[{\normalfont (A)}]
\item If $f\in \textrm{BMO}$,  then $T_f$ is bounded on $F^p_\varphi$ if and only if $\widetilde{f}$ is bounded on $\mathbb{C}^n$;
 $T_f$ is compact on $F^p_\varphi$ if and only if
\begin{equation}
\label{e:4.3}
 \lim\limits_{z\to \infty} \sup\limits_{w\in B(z, r)}
   \left|\langle T_f k_z, k_w \rangle_{F^2_\varphi}\right|= 0\quad \forall r>0.
 \end{equation}

\item If $S\in {\mathcal B}_p^{\varphi}$, then $S$ is compact on $F^p_\varphi$  if and only if
\begin{equation}
\label{e:4.4}
\lim\limits_{z\to \infty} \sup\limits_{w\in B(z, r)} \left|\langle S k_z, k_w \rangle_{F^2_\varphi}\right|= 0\quad\forall r>0.
\end{equation}
\end{enumerate}
\end{thm}

\begin{proof}
We claim  $T_f\in \textrm{WL}_p^\varphi$ if $f\in \textrm{BMO}$ and $\widetilde{f}$ remains bounded. In fact,   similar to \cite{CIL11}*{Lemma 1}  it is trivial to verify
\begin{equation}
\label{e:4.5}
\sup\limits_{z\in \mathbb{C}^n}\widetilde{|f|}(z)\leq \|f\|_{\textrm{BMO}}+\sup\limits_{z\in \mathbb{C}^n}|\widetilde{f}(z)|<\infty.
\end{equation}
By \cite{HL14}*{Theorem 3.5},  $|f|dV$ is a Fock-Carleson measure. Hence,
  \beqm
 && \left|\langle T_f k_z, k_w\rangle_{F^2_\varphi}\right|\leq\int_{\mathbb{C}^n}\left|k_z(u)k_w(u)\right|e^{-2\varphi(u)}|f(u)|dV(u)
  \\
  &&\verb#           #\leq C\sup\limits_{u\in\mathbb{C}^n}\widehat{|f|}_r(u)\int_{\mathbb{C}^n}\left|k_z(u)k_w(u)\right|e^{-2\varphi(u)}dV(u)
  \\
  &&\verb#           #\leq C\sup\limits_{u\in\mathbb{C}^n}\widetilde{{|f|}}(u)\int_{\mathbb{C}^n}e^{-\theta|z-u|-\theta|w-u|}dV(u)
  \\
  &&\verb#           #\leq Ce^{-\frac{\theta}{2}|z-w|}.
\eqm
This implies  $T_f\in \textrm{WL}_p^\varphi$ for any $p\in (0, \infty)$ (and also, $T_f$ is strongly localized in the sense of Xia and Zheng, see \cite{XZ13}).

(A).  Suppose $f\in \textrm{BMO}$.  If $\widetilde{f}$ is bounded on $\mathbb{C}^n$,
then $T_f\in \textrm{WL}_p^\varphi$ which implies $T_f$ is bounded on $F^p_\varphi$ for any $p\in (0, \infty)$.  Conversely, the condition that $T_f$ is bounded implies $\widetilde{f}$ is bounded, which can be proved in a standard way with $\widetilde{f}(z)= \langle T_f k_z, k_z\rangle_{F^2_\varphi}$.

  Now we deal with the compactness of $T_f$. If \eqref{e:4.3} holds, then $\widetilde{f}(z)=\langle T_f k_z, k_z \rangle_{F^2_\varphi}$ is bounded, hence
 $T_f\in \textrm{WL}_p^\varphi$. Therefore, by \eqref{e:4.3} and Theorem \ref{t:3.7},  $T_f$ is compact  on $F^p_\varphi$ for all $0<p<\infty$. Conversely, if $T_f$ is compact  on $F^p_\varphi$ for some $0<p<\infty$.
 If $1<p<\infty$,  we have $\lim\limits_{z\to \infty}\|T_f k_z \|_{p, \varphi} = 0$ because $k_z$ tends to zero weakly, from which \eqref{e:4.3} follows for any $r>0$.   If $0<p\le 1$, we know $\widetilde{f}$ to be bounded.  Then, $T_f\in \textrm{WL}_p^\varphi$. Now the estimate \eqref{e:4.3} comes from Theorem \ref{t:3.7}.

(B) Since each $T_{f_j}\in \textrm{WL}_p^\varphi$,  we have  ${\mathcal B}_p^{\varphi}\subset\textrm{WL}_p^\varphi$ for  $0<p<\infty$. Now  the conclusion  follows from Theorem \ref{t:3.7}.
\end{proof}

 As shown by Isralowitz in  \cite{Is13}*{Proposition 1.5}, on the classical Fock space $F^p_\alpha$ the estimate \eqref{e:4.4} is equivalent to $\lim\limits_{z\rightarrow\infty}\widetilde{S}(z)=0$. Therefore, Theorem \ref{t:4.2} extends  \cites{CIL11, HL14}.

  As in \cite{Is11}, set $\textrm{BT}$ to be the collection of all measurable functions $f$ on $\mathbb{C}^n$ with $\widetilde{|f|}$ bounded. As shown in the proof of Theorem 4.2, $T_f\in {\textrm{WL}}_p^\varphi$ if $f\in \textrm{BT}$.   We have  Corollary \ref{c:4.3} at once.

\begin{cor}
\label{c:4.3}
Let $0<p<\infty$, and let $S$ be in  the family  all linear combination of  the form $T_{f_1} T_{f_2}\cdots T_{f_m}$, where  each  function $f_j\in$BT. Then $S$ is compact on $F^p_\varphi$ if and only if one of the following three statements holds:
\begin{enumerate}[{\normalfont (A)}]

\item $ \lim\limits_{z\to \infty} \sup\limits_{w\in B(z, r)}
   \left|\langle Tk_z, k_w \rangle\right|= 0$ for any    $r>0$;

 \item $ \lim\limits_{z\to \infty} \sup\limits_{w\in \mathbb{C}^n}
   \left|\langle Tk_z, k_w \rangle\right|= 0$;

\item  $ \lim\limits_{z\to \infty} \|T k_z\|_{p, \varphi} =0$.
\end{enumerate}
\end{cor}

 While  $\varphi(z)= \fr 14 |z|^2$, $p=2$ and $S$ is a linear combination of operators of the form $T_{f_1} T_{f_2}\cdots T_{f_m}$ with each $f_j\in \textrm{BT}$, Corollary \ref{c:4.3} gives the main result of \cite{Is11}.

\section{Operators Satisfying Axler and Zheng's Condition}
 \label{s:four}

In this section, we will restrict ourselves to  the classical Fock space $F^p_\alpha$, that is  $\varphi(z)=\fr \alpha 2 |z|^2$ with $\alpha>0$. We are going to characterize the boundedness and compactness of linear operators with the Axler-Zheng condition  on $F^p_\alpha$.

Let  $\phi_z$  be  the holomorphic self-map of $\mathbb{C}^n$,  $\phi_z(\cdot)= z-\cdot $.
  $U_z$ is the operator on $F^p_\alpha$ defined by $U_z f= (f\circ \phi_z) k_{z}$.
   Given some linear operator $S$ on $F^p_\alpha$, define
   $$
   S_z= U_z S U_z^*.
   $$
In the context  of Bergman space $A^2(\mathbb{D})$  on the unit disc $\mathbb{D}$, with $\phi_z(w)= \fr {w-z}{ 1-\overline{z}w}$ and $U_z f= (f\circ \phi_z)\phi'(z)$,  Axler-Zheng introduced the condition
$$
\sup_{z\in \mathbb{D}} \|S_z1\|_{A^p }<\infty \textrm{ with some } p>2
$$
 in  \cite{AZ98}.  The  work
    in   \cites{CIL11, En99, MZ04, MW12, Zo03} also explored the condition $\|S_z 1\|_{A^p }\le C$. In the Fock space setting,   Wang, Cao, and Zhu  carried out related research  in \cite{WCZ12} to obtain that, if there exist some $p > 2$ such that
    $$
    \sup_{z\in \mathbb{C}^n} \|S_z 1\|_{p,\frac{2\alpha}{p}}<\infty \verb#  #
     ( \textrm{or } \|S_z 1\|_{p,\frac{2\alpha}{p}}\rightarrow 0
     \textrm {  as } z\rightarrow\infty),
 $$
 the operator $S$ is bounded (or compact) on $F^2_\alpha$.


\begin{thm}
\label{t:4.4}
Suppose $S$ is a linear operator defined on $\mathcal D$.  If there are some  $0<\sigma<p<\infty$ such that
\begin{equation}
\label{e:4.6}
 M= \sup_{z\in \mathbb{C}^n}  \int_{\mathbb{C}^n} |S_z1(u)|^{p} e^{-\fr {\alpha\sigma}2 |u|^2} dV(u)<\infty,
\end{equation}
  then
$$
\left| \langle Sk_{z}, k_{w}\rangle_{F^2_\alpha}    \right |   \le C M^{\fr 1{p}} e^{- \fr {\alpha(p-\sigma)} {2p}  |z-w|^2},
$$
so $S$ is bounded on $F^s_\alpha$ for all $0<s< \infty$. Furthermore, if both \eqref{e:4.6}  and
\begin{equation}
\label{e:4.7}
   \lim_{z\to \infty} \int_{\mathbb{C}^n} |S_z1(u)|^{p} e^{-\fr {\alpha\sigma}2 |u|^2} dV(u) =0
\end{equation}
hold, then $S$ is compact on $F^s_\alpha$ for  $0<s<\infty$.
\end{thm}

\begin{proof}
Since $K(\cdot, \cdot)=e^{\alpha\langle \cdot, \cdot\rangle}$, it is easy to verify
$k_{z}(z-u) k_{z}(u)=1$ and $$ K(z-u, z-u)=K(z, z)K(u, u)|K(u, z)|^{-2}.$$ By the equality
   $ S_z1 (u)= k_{z}(u) (Sk_{z})(z-u)$ (see \cite{Zh12}) and Lemma \ref{l:2.3} we have
\beqm
\int_{\mathbb{C}^n} |S_z1(u)|^{p} e^{-\fr {\alpha\sigma}2 |u|^2} dV(u) & = & \int_{\mathbb{C}^n}  | k_{z}(u) (Sk_{z})(z-u)|^{p} e^{-\fr {\alpha\sigma}2 |u|^2} dV(u)\\
    & = &  \int_{\mathbb{C}^n} |k_{z}(z-u) (Sk_{z})( u)|^{p} e^{-\fr {\alpha\sigma}2 |u-z|^2} dV(u)\\
    & = &  \int_{\mathbb{C}^n} |  (Sk_{z})( u)|^{p} |k_{z}(u)|^{-p} e^{-\fr {\alpha\sigma}2 |u-z|^2}  dV(u)\\
    &\geq &   \int_{B(w, 1)} |  (Sk_{z})( u)|^{p} |k_{z}(u)|^{-(p-\sigma)}  e^{-\fr {\alpha\sigma}2 |u|^2} dV(u)\\
    & \geq &   C  |  (Sk_{z})(w)|^{p} |k_{z}(w)|^{-(p-\sigma)}  e^{-\fr {\alpha\sigma}2 |w|^2}\\
     & = &  C  | \langle  Sk_{z} , k_{w} \rangle _{F^2_\alpha}|^{p}e^{ \fr {\alpha (p-\sigma)} 2 |z- w|^2}.\verb#   #
\eqm
 From the above inequalities
and  \eqref{e:4.6},  we get
 $$
          | \langle  Sk_{z} , k_{w}\rangle_{F^2_\alpha}|\le C  M^{\fr 1{p}} e^{- \fr {\alpha(p-\sigma)}{ 2p}  |z- w|^2} .
$$
Since $p-\sigma>0$,   $S$ is weakly localized for $F^s_\alpha$, so
 $S$ is bounded on $F^s_\alpha$ for all $0<s<\infty$.

Furthermore, if both \eqref{e:4.6} and \eqref{e:4.7} are valid, from the proof above we have $ S\in \textrm {WL}_s^\alpha$. And also, for $p\in (0, \infty)$ there is some constant $C_r$ such that
$$
  \sup_{w\in B(z, r)}  \left | \langle Sk_{z}, k_{z}\rangle_{F^2_\alpha}\right| \le C_r \left( \int_{\mathbb{C}^n} |S_z1(u)|^{p} e^{-\fr {\alpha\sigma}2 |u|^2}
    dV(u)\right)^{\fr 1{p}}\to 0
$$
as
$ z \to \infty$.  By Theorem \ref{t:3.7} $S$ is compact on $F^s_\alpha$ for all $0<s<\infty$.
\end{proof}

\textbf{Remark. } If $p>\sigma=2$ and $s=2$, then Theorem \ref{t:4.4} reduces to Theorems A and B in \cite{WCZ12}.

\section{Further  Remarks }
\label{s:four}

An important theme in analysis on function spaces is to characterize when a given operator is compact.  In the setting of the Bergman space $A^p(\mathbb{B}_n)$ on the unit ball $\mathbb{B}_n$, $1<p<\infty$, in 2007 Su\'arez proved, see \cite{Su07}, that a bounded operator $S$ is compact if and only if $S$ is in the Toeplitz algebra and the Berezin transform of $S$ vanishes on the boundary. Later on, Mitkovski, Su\'arez and  the third author \cite{MSW13} extended \cite{Su07} to the weighted  Bergman space $A^p_{\alpha}(\mathbb{B}_n)$. On the classical Fock space $F^p_\alpha$ for $1<p<\infty$, in \cite{BI12} Bauer and Isralowitz showed that Su\'arez's characterization on compact operators is valid. For general $\varphi$ with $dd^c \varphi \simeq \omega_0$ and $1<p<\infty$, most recently in \cite{Is13} Isralowitz obtained ${\mathcal K}(F^p_\varphi)={\mathcal T}^p_\varphi (C^\infty_c(\mathbb{C}^n))$, which implies the results in \cite{BI12}.

For Toeplitz operators $T_\mu$ with positive Borel measures $\mu$ as symbols, the boundedness (or compactness) on $F^p_\varphi$ with $0<p\le 1$ can be characterized with the same condition as that on $F^q_\varphi$ with $q>1$. Unfortunately, some differences appear when we talk about the structure of ${\mathcal K}(F^p_\varphi)$. For example,  we find  $ {\mathcal K}(F^p_\varphi) \backslash {\mathcal T}^p_\varphi\neq \emptyset$ if $0<p\le 1$.
To see this, from \cite{Zh12}*{Lemma 4.39} (or \cite{MMO03}) we take a separated sequence $\{z_j\}_{j=1}^\infty $ which is an interpolating sequence for $F^\infty_\alpha$. Hence, we have some  $f\in F^\infty_\alpha$ such that
 \begin{equation}
 \label{e:5.1}
 f(z_k)e^{-\frac{\alpha}{2}|z_k|^2}=1,\verb#     #\forall k\in \mathbb{N}.
 \end{equation}
Although  \cite{Zh12} is only concerned with one variable interpolation,  take $\{z_j\}\subset \mathbb{C}$   and  $f\in H(\mathbb{C})$   satisfying  the interpolation above,   extend $f$ to $\mathbb{C}^n$ with the equation $f(z, z')= f(z)$ for $(z, z')\in {\mathbb{C}} \times {\mathbb{C}}^{n-1}$, we will have $f$ satisfying \eqref{e:5.1} in $\mathbb{C}^n$.
Furthermore, for $0<p\le 1$  take  $g\in F^p_\alpha$ so that $g(0)\neq 0$.  Define the operator $T$ on $F^p_\alpha$ as
\begin{equation}
\label{e:5.2}
T(\cdot)=\langle \cdot, f\rangle_{F^2_\alpha} g.
\end{equation}
$T$ is bounded and   of rank $1$, so $T$ is compact on $F^p_\alpha$. Also,
$$
 \left|\langle T k_{z_k}, k_w \rangle_{F^2_\alpha}\right|= \left|\overline{f(z_k)}e^{-\frac{\alpha}{2}|z_k|^2}g(w)e^{-\frac{\alpha}{2}|w|^2}\right|.
$$
 Because $\{z_j\}_{j=1}^\infty $ is separated, we have
$
        \lim\limits_{j\to \infty} z_j =\infty
$.
For each $r>0$,  as $k$ is large enough  we have from Lemma \ref{l:2.3} that
\beqm
&& \int_{B(z_k,r)^c} \left|\langle T k_{z_k}, k_w \rangle_{F^2_\alpha}\right|^p
  dV(w)=\int_{B(z_k,r)^c} \left|g(w)e^{-\frac{\alpha}{2}|w|^2}\right|^p
  dV(w)
  \\&&\verb#                      #\geq\int_{B( 0,1)} \left|g(w)e^{-\frac{\alpha}{2}|w|^2}\right|^p
  dV(w)
  \\&&\verb#                      #\geq C \left|g( 0)\right|^p.
\eqm
$T \notin \textrm{WL} _p^\alpha$ by Definition \ref{d:2.1}. Hence,  $ {\mathcal K}(F^p_\alpha) \backslash {\mathcal T}^p_\alpha\neq \emptyset$  for $0<p\le 1$. This tells us the characterization  of  compact operators $T$   on $F^p_\varphi$ with $0<p\le 1$ is quite different from that with $1<p<\infty$.   \newline

For $0<p\leq 1$,  $\{k_z: z\in \mathbb{C}^n\}$ does not converge weakly to zero in $F^p_\alpha$  as $z$ goes to $\infty$. In fact, take  $f \in F^\infty_\alpha$ satisfying \eqref{e:5.1},  since the dual space of $F^p_\alpha$ is $F^\infty_\alpha$ under the pairing $\langle g, f \rangle_{F^2_\alpha}$ (see \cite{Zh12}), we know that $F_f=\langle \cdot, f \rangle_{F^2_\alpha}$ is a bounded linear functional on $F^p_\alpha$. However,
$$
F_f(k_{z_k})=\langle k_{z_k},f \rangle_{F^2_\alpha}=1
$$
for all $k$.  \newline

 The operator $T$ defined as \eqref{e:5.2} also shows   $\lim\limits_{z\to \infty }\|Tk_z\|_{p, \alpha} \neq 0$, because  $ Tk_{z_j}  =  g $ for $j=1, 2, \cdots$, which says $T k_z $ need not converge to $0$ in $F^p_\alpha$ even if $T$  is compact while $0<p\le 1$. So, the hypothesis $T\in {\textrm WL}^\varphi_p$ both in Theorem \ref{t:3.7} and \ref{t:3.8} can not be removed.   But for the Berezin transform, we  have the following Proposition \ref{p:5.1}.

\begin{prop}
\label{p:5.1}
Suppose $0<p\le 1$ and $T\in {\mathcal K}(F^p_\varphi)$. Then  $\widetilde{T}(z)\rightarrow 0$ as $z\rightarrow\infty$.
\end{prop}

\begin{proof}
For  $R>0$ fixed,   $P M_{\chi_{B(0,R)}}Tk_z\in F^p_\varphi\subset F^2_\varphi$. Lemma \ref{l:2.2} estimate \hyperlink{l:2.2e1}{(1)} and Lemma \ref{l:2.2} estimate \hyperlink{l:2.2e3}{(3)} give
 \beqm
 \left|\langle P M_{\chi_{B(0,R)}}Tk_z, k_z \rangle_{F^2_\varphi}\right| & = & \left|\langle  M_{\chi_{B(0,R)}}Tk_z, k_z\rangle_{F^2_\varphi}\right|
      \leq  \int_{B(0,R)}  \left| Tk_z(u) \overline{k_z(u)}\right|e^{-2\varphi(u)}dV(u) \\
    &&\verb#                 # \le \|Tk_z\|_{\infty, \varphi} \int_{B(0,R)}  \left| k_z(u)\right|e^{-\varphi(u)}dV(u)\\
      &&\verb#                 #  \leq C \|Tk_z\|_{p, \varphi}  \sup\limits_{|u|\leq R}  \left| k_z (u)\right|e^{-\varphi(u)}\\
       && \verb#                 #  \leq\|T\|_{F^p_\varphi\rightarrow F^p_\varphi}\|k_z\|_{p, \varphi} e^{-\theta|z|}\\
        && \verb#                 # \leq C e^{-\theta|z|}\rightarrow0
     \eqm
 as $z\rightarrow\infty$. Since  $T\in {\mathcal K}(F^p_\varphi)$,  Lemma \ref{l:3.5} tells us
  \beqm
    && \left|\left\langle \left(T-P M_{\chi_{B(0,R)}}T\right)k_z, k_z\right \rangle_{F^2_\varphi}\right|\leq\left\| \left(T-P M_{\chi_{B(0,R)}}  T\right)k_z\right\|_{\infty, \varphi}\|k_z\|_{1, \varphi}\\
     &&\verb#                          # \leq C \left\| \left(T-P M_{\chi_{B(0,R)}}T\right)k_z\right\|_{p, \varphi}
     \\
    &&\verb#                          #  \le C\|T-P M_{\chi_{B(0,R)}}T \|_{F^p_\varphi\rightarrow F^p_\varphi}\|k_z\|_{p, \varphi} ,
    \\
    &&\verb#                          #  \le C\|T-P M_{\chi_{B(0,R)}}T \|_{F^p_\varphi\rightarrow F^p_\varphi}\rightarrow0
     \eqm
as $R\rightarrow\infty$. Therefore, taking $z\rightarrow\infty$,
\beqm
    && |\widetilde{T}(z)|=\left|\langle Tk_z, k_z \rangle_{F^2_\varphi}\right|\leq \left|\left\langle P M_{\chi_{B(0,R)}}Tk_z, k_z\right \rangle_{F^2_\varphi}\right|+ \left|\left\langle \left(T-P M_{\chi_{B(0,R)}}T\right)k_z, k_z\right \rangle_{F^2_\varphi}\right|\rightarrow0.
    \eqm
\end{proof}

Summarizing the discussion above we put forward the following problem.

\begin{prob}
\label{prob:1}
For $0<p\le 1$, what are the necessary and sufficient conditions to characterize the membership  in ${\mathcal K}(F^p_\varphi)$?
\end{prob}

Under the restriction $0<p\le 1$, we dominate the essential norm of $T\in\textrm{WL}_p^\varphi$  by its behavior on  $k_z$, see Theorem \ref{t:3.8}.   Our second problem is whether the estimate \eqref{e:3.9}  still holds for $1<p<\infty$?.

\begin{prob}
\label{prob:2}
Suppose  $1<p<\infty$. Does
$$
\|T_f\|_{e, F^p_\varphi}\simeq \limsup_{z \to \infty} \|T_fk_z\|_{p, \varphi}
$$
hold for  bounded $f$ on $\mathbb{C}^n$?
\end{prob}

 In the previous  section,  with  $f\in \textrm{BMO}$ we have  obtained   the  compactness of Toeplitz operators $T_f$ on $F^p_{\varphi}$.
  However, to consider the compactness of finite product $T_{f_1} T_{f_2}\cdots T_{f_m}$ of Toeplitz operators with $\textrm{BMO}$ symbols we have assumed each symbol $f_j$ has a bounded Berezin transform. Is this hypothesis necessary in the statement (B) of Theorem \ref{t:4.2}?

\begin{prob}
\label{prob:4}
Suppose $0<p<\infty$, and $T$ is in the set of  all linear combination of  the form $T_{f_1} T_{f_2}\cdots T_{f_m}$, where  each  function $f_j\in \textrm{BMO}$. Can we conclude that  $T$ is compact on $F^p_\varphi$  if and only if
$$
\lim\limits_{z\to \infty} \sup\limits_{w\in B(z, r)} \left|\langle T k_z, k_w \rangle_{F^2_\varphi}\right|= 0
$$
holds for each $ r>0$?
\end{prob}

We also point to the general question of how the story is similar, or different, in the case of the Bergman space $A^p(\mathbb{D})$ when $0<p<1$.\newline

{\bf Acknowledgments. }  The authors would like to thank the referees for making some good suggestions.




\begin{bibdiv}
\begin{biblist}
\normalsize

\bib{AZ98}{article}{
   author={Axler, Sheldon},
   author={Zheng, Dechao},
   title={Compact operators via the Berezin transform},
   journal={Indiana Univ. Math. J.},
   volume={47},
   date={1998},
   number={2},
   pages={387--400}
}

\bib{BI12}{article}{
   author={Bauer, Wolfram},
   author={Isralowitz, Joshua},
   title={Compactness characterization of operators in the Toeplitz algebra
   of the Fock space $F^p_\alpha$},
   journal={J. Funct. Anal.},
   volume={263},
   date={2012},
   number={5},
   pages={1323--1355}
}

\bib{CIJ14}{article}{
   author={Cho, Hong Rae}
   author={Isralowitz, Josh},
   author={Joo, Jae-Cheon},
   title={Toeplitz operators on Fock-Sobolev type spaces},
   journal={Integral Equations Operator Theory},
   volume={82},
   date={2015},
   number={1},
   pages={1--32}
}

\bib{CZ12}{article}{
   author={Cho, Hong Rae},
   author={Zhu, Kehe},
   title={Fock-Sobolev spaces and their Carleson measures},
   journal={J. Funct. Anal.},
   volume={263},
   date={2012},
   number={8},
   pages={2483--2506}
}

\bib{CIL11}{article}{
   author={Coburn, L. A.},
   author={Isralowitz, Josh},
   author={Li, Bo},
   title={Toeplitz operators with $\textrm{BMO}$ symbols on the Segal-Bargmann space},
   journal={Trans. Amer. Math. Soc.},
   volume={363},
   date={2011},
   number={6},
   pages={3015--3030}
}



\bib{En99}{article}{
   author={Engli{\v{s}}, Miroslav},
   title={Compact Toeplitz operators via the Berezin transform on bounded
   symmetric domains},
   journal={Integral Equations Operator Theory},
   volume={33},
   date={1999},
   number={4},
   pages={426--455}
}

\bib{HL11}{article}{
   author={Hu, Zhangjian},
   author={Lv, Xiaofen},
   title={Toeplitz operators from one Fock space to another},
   journal={Integral Equations Operator Theory},
   volume={70},
   date={2011},
   number={4},
   pages={541--559}
}

\bib{HL14}{article}{
   author={Hu, Zhangjian},
   author={Lv, Xiaofen},
   title={Toeplitz operators on Fock spaces $F^p(\varphi)$},
   journal={Integral Equations Operator Theory},
   volume={80},
   date={2014},
   number={1},
   pages={33--59}
}

\bib{Is11}{article}{
   author={Isralowitz, Josh},
   title={Compact Toeplitz operators on the Segal-Bargmann space},
   journal={J. Math. Anal. Appl.},
   volume={374},
   date={2011},
   number={2},
   pages={554--557}
}

\bib{Is13}{article}{
   author={Isralowitz, Josh},
   title={Compactness and essential norm properties of operators on generalized Fock spaces},
   journal={J. Operator Theory},
   volume={73},
   date={2015},
   number={2},
   pages={281--314}
  }

\bib{IMW13}{article}{
   author={Isralowitz, Josh},
   author={Mitkovski, Mishko},
   author={Wick, Brett D.},
   title={Localization and compactness in Bergman and Fock spaces},
   journal={Indiana Univ. Math J.},
   volume={64},
   date={2015},
   number={5},
   pages={1553--1573}
}

\bib{MMO03}{article}{
   author={Marco, N.},
   author={Massaneda, X.},
   author={Ortega-Cerd{\`a}, J.},
   title={Interpolating and sampling sequences for entire functions},
   journal={Geom. Funct. Anal.},
   volume={13},
   date={2003},
   number={4},
   pages={862--914}
}

\bib{MZ04}{article}{
   author={Miao, Jie},
   author={Zheng, Dechao},
   title={Compact operators on Bergman spaces},
   journal={Integral Equations Operator Theory},
   volume={48},
   date={2004},
   number={1},
   pages={61--79}
}

\bib{MSW13}{article}{
   author={Mitkovski, Mishko},
   author={Su{\'a}rez, Daniel},
   author={Wick, Brett D.},
   title={The essential norm of operators on $A^p_\alpha(\mathbb{B}_n)$},
   journal={Integral Equations Operator Theory},
   volume={75},
   date={2013},
   number={2},
   pages={197--233}
}

\bib{MW12}{article}{
   author={Mitkovski, Mishko},
   author={Wick, Brett D.},
   title={A reproducing kernel thesis for operators on Bergman-type function
   spaces},
   journal={J. Funct. Anal.},
   volume={267},
   date={2014},
   number={7},
   pages={2028--2055}
}

\bib{SV12}{article}{
   author={Schuster, Alexander P.},
   author={Varolin, Dror},
   title={Toeplitz operators and Carleson measures on generalized
   Bargmann-Fock spaces},
   journal={Integral Equations Operator Theory},
   volume={72},
   date={2012},
   number={3},
   pages={363--392}
}

\bib{Su07}{article}{
   author={Su{\'a}rez, Daniel},
   title={The essential norm of operators in the Toeplitz algebra on $A^p(\mathbb{B}_n)$},
   journal={Indiana Univ. Math. J.},
   volume={56},
   date={2007},
   number={5},
   pages={2185--2232}
}


 \bib{WX16}{article}{
   author={Wang, Chunjie},
     author={Xiao, Jie},
      title={Addendum to "Gaussian integral means of entire functions"},
      journal={Complex Anal. Oper. Theory},
       volume={10},
         date={2016},
         number={3},
          pages={495--503}
          }

\bib{WCZ12}{article}{
   author={Wang, Xiaofeng},
   author={Cao, Guangfu},
   author={Zhu, Kehe},
   title={Boundedness and compactness of operators on the Fock space},
   journal={Integral Equations Operator Theory},
   volume={77},
   date={2013},
   number={3},
   pages={355--370}
}

\bib{XZ13}{article}{
   author={Xia, Jingbo},
   author={Zheng, Dechao},
   title={Localization and Berezin transform on the Fock space},
   journal={J. Funct. Anal.},
   volume={264},
   date={2013},
   number={1},
   pages={97--117}
}

\bib{Zh12}{book}{
   author={Zhu, Kehe},
   title={Analysis on Fock spaces},
   series={Graduate Texts in Mathematics},
   volume={263},
   publisher={Springer, New York},
   date={2012},
   pages={x+344}
}

\bib{Zo03}{article}{
   author={Zorboska, Nina},
   title={Toeplitz operators with $\textrm{BMO}$ symbols and the Berezin transform},
   journal={Int. J. Math. Math. Sci.},
   date={2003},
   number={46},
   pages={2929--2945}
}

\end{biblist}
\end{bibdiv}
\end{document}